\documentclass[12pt]{amsart}
\usepackage{a4wide}
\usepackage{mathrsfs}

\usepackage{graphicx}
\usepackage[abs]{overpic}
\usepackage{here}
\usepackage{contour}

\newcommand{\arc}[1]{%
 \settowidth{\dimen0}{\ensuremath{#1}}%
 \divide\dimen0 by 2%
 \overset{\rotatebox{-90}{\ensuremath{\left(\rule{0pt}{\dimen0}\right.}}}{#1}%
}

\usepackage[%
 setpagesize=false,%
 bookmarks=true,%
 bookmarksdepth=tocdepth,%
 bookmarksnumbered=true,%
 colorlinks=false,%
 pdftitle={},%
 pdfsubject={},%
 pdfauthor={},%
 pdfkeywords={}%
]{hyperref}

\usepackage{amsmath,amssymb,amsthm,mathtools}
\usepackage{tikz}
\usetikzlibrary{intersections,calc,arrows.meta}

\newtheorem*{annulus decomposition theorem}{Annulus decomposition theorem}
\theoremstyle{definition}
\newtheorem{definition}{Definition}[section]
\newtheorem*{theorem}{Theorem}
\newtheorem{introtheorem}{Theorem}

\newtheorem{proposition}[definition]{Proposition}
\newtheorem*{proposition*}{Proposition}
\newtheorem{lemma}[definition]{Lemma}
\newtheorem{corollary}[definition]{Corollary}
\newtheorem{remark}[definition]{Remark}
\newtheorem*{acknowledgement}{Acknowledgement}

\title{GRAFTING OF REAL PROJECTIVE SURFACES WITH HITCHIN HOLONOMY}
\author{Toshiki Fujii}
\address{Graduate School of Science, Osaka University}
\email{u338912j@ecs.osaka-u.ac.jp}

\begin{document}
  \maketitle

  \begin{abstract}
    We define graftble curves on real projective surfaces. In particular, we construct graftbale ones in Hitchin case and show that real projective structures with the same Hitchin holonomy, carrying the same weight type, are related to each other via multi-graftings.
  \end{abstract}
  
  \section{Introduction}

  Let $\Sigma$ be an oriented closed smooth surface with genus $ g > 1$, and let $\widetilde{\Sigma}$ denote the universal cover of $\Sigma$. A \textit{real projective structure} on $\Sigma$ is a $(\mathbb{RP}^2,PGL(3,\mathbb{R}))$-structure. Namely, it is the maximal atlas of $\mathbb{RP}^2$ with transition maps in $PGL(3,\mathbb{R})$.

  There is another equivalent definition. A real projective structure is a pair $(D,\rho)$, where $D \colon \widetilde{\Sigma} \to \mathbb{RP}^2$ is a local diffeomorphism and $\rho \colon \pi_1(\Sigma) \to PGL(3,\mathbb{R})$ is a representation such that $D$ is $\rho$-equivariant, i.e., $D \circ \gamma = \rho(\gamma) \circ D$ for all $\gamma \in \pi_1(\Sigma)$. The map $D$ is called the \textit{developing map} and $\rho$ the \textit{holonomy representation} (see \cite{Goldman:geo}). Two projective structures $(D_1,\rho_1)$ and $(D_2,\rho_2)$ are equivalent if there exists a $\pi_1(\Sigma)$-equivariant diffeomorphism $\Phi \colon \widetilde{\Sigma} \to \widetilde{\Sigma}$ and $A \in PGL(3,\mathbb{R})$ such that $D_2 = A \circ D_1 \circ \Phi$ and $\rho_2 = A \circ \rho_1 \circ A^{-1}$.

  In this paper, we consider an operation called \textit{grafting} that produces a new one from given a real projective structure without changing the holonomy representation. More concretely, grafting operation proceeds by cutting a projective surface along a particular type of simple closed curve, called \textit{graftable curve}, and inserting a special type of annulus (e.g., a Hopf annulus) along the curve (see \S 2.4). 

  Grafting was used by D.Sullivan and W.Thurston \cite{ST:exa} to produce examples of real projective structures on the two torus, whose developing maps fails to be covering maps onto its images. Traditionally, graftable curves on $\mathbb{RP}^2$-surfaces have been taken to be simple closed geodesics, namely those that develop into straight line segments in $\mathbb{RP}^2$. In this paper, we introduce a topological definition of graftable curves.
  
  Grafting is important because real projective structures with Hitchin hololonomy bijectively correspond to \textit{grafting data} consisting a pair $(\{\alpha_i\}_{i=1}^n,\{w_i\}_{i=1}^n)$, where $\{\alpha_i\}_{i=1}^n$ is the finite collection of disjoint simple closed curves and $\{w_i\}_{i=1}^n$ is the collection of specific weights. Namely, these weights are words written as products of even exponent in two elements $x$ and $y$, called \textit{completely even words}. Goldman proved that every real projective structure on $\Sigma$ with Hitchin holonomy is obtained from its associated convex real projective structure by grafting an annulus corresponding to $w_i$ along each $\alpha_i$ (see \cite{Goldman:fuc}). This operation is called a \textit{multi-grafting}. 

  Similarly to the real projective strcutres above, every $\mathbb{CP}^1$-structure on $\Sigma$ with Fuchsian holonomy is obtained by grafting its associated hyperbolic structure along a weighted multi-curve. The weights take value in $\mathbb{Z}_{>0}$ and they correspond to Hopf-annuli available for grafting. Consequently, the orientation of the multi-curve is irrelevant. By contrast, in the case of real projective structures, the weights are more intricate, and one has to care for the orientation of the multi-curve of grafting data ($\{\alpha_i\}_{i=1}^n$,$\{w_i\}_{i=1}^n$): inserting annulus corresponding to $w_i$ along $\alpha_i$ is equivalent to inserting the annulus corresponding to some weight ${w_i}^{*}$ along ${\alpha_i}^{-1}$. Algebraically, ${w_i}^{*}$ can be also described as the word obtained by reversing $w_i$ and interchanging $x$ and $y$. We define a \textit{weight type} of grafting data ($\{\alpha_i\}_{i=1}^n,\{w_i\}_{i=1}^n$) as a set of the weights $\{w_1,{w_1}^{*}, \cdots ,w_n,{w_n}^{*}\}$ (see Definition \ref {def:5.1}).
  
  In the setting of $\mathbb{CP}^1$-structures, it is proved in \cite{CDF:gra} that one can pass between any two exotic projective structures with Fuchsian holonomy $\rho$ by at most two compositions of multi-graftings. In this paper, we obtain the following analogous result for real projective structures:

  \begin{introtheorem}
    Let $\sigma_1$ and $\sigma_2$ be two real projective structures on $\Sigma$ with the same Hitchin holonomy. Suppose in addition that they correspond to the grafting data with the same weight type. Then, $\sigma_2$ can be obtained from $\sigma_1$ by a composition of at most $6g$ multi-graftings.
  \end{introtheorem}
 
 As a direct consequence, given an exotic projective structure, one can return to it by repeating multi-graftings. To prove the result, we concretely construct graftable curves on real projective surfaces with Hitchin holonomy $\rho$. This construction is carried out using the method similar to that in \cite{CDF:gra}. In the process, we obtained the following result (see Proposition \ref{prop:3.2}).

 \begin{introtheorem}
 Let $M$ be a real projective surface with Hitchin holonomy $\rho$. For any essential simple closed curve $\beta$, there is a graftable simple closed curve on $M$ isotopic to $\beta$. 
 \end{introtheorem}

 Let $MG(\rho)$ be the oriented graph whose vertices are the projective structures with Hitchin holonomy $\rho$ and two vertices $\sigma_1$ and $\sigma_2$ are joined by an oriented edge from $\sigma_1$ to $\sigma_2$ if $\sigma_2$ is obtained from $\sigma_1$ by a multi-grafting. The main result implies that, given two vertices $\sigma_1$ and $\sigma_2$ of $MG(\rho)$, if the weight type of $\sigma_1$ is ``contained'' (in a certain sense) in the weight type of $\sigma_2$, they are joined from $\sigma_1$ to $\sigma_2$ (see Corollary \ref{cor:6.1}).  
 
 We use some surgery operations on multi-curves introduced by Luo \cite{Luo:multi} and later closely observed by Ito \cite[\S 2.5]{Ito:exo}. In this paper, however, we introduce new notation in \S 4.1 to avoid unnecessarily complexity. (See Remark \ref{rem:4.1}.)  

  This paper is organized as follows. In \S 2, we discuss the definition of graftable curves and the grafting process. In \S 3, we construct graftable curves on non-convex real projective surfaces with Hitchin holonomy. This construction yields Theorem B. In \S 4, we describe the projective structures obtained by grafting a non-convex real projective structure along the curves constructed in \S 3. In \S 5, we prove Theorem A. In \S 6, we construct a model of $MG(\rho)$ associated with weight types that reflects part of its connectivity.

  \begin{acknowledgement}
    The author would like to thank for my supervisor, Shinpei Baba, for many discussion and helpful advice.
  \end{acknowledgement}
  
  \section{Graftable curves}

  We define graftable curves on general $\mathbb{RP}^2$-surfaces and describe the grafting process. In the case of $\mathbb{CP}^1$-structures, grafting is introduced in \cite[\S 2]{CDF:gra} and \cite[\S 1]{Goldman:fuc}. The basic concepts described in \S 2.1 are based on \cite[\S 3]{Goldman:fuc}.

\subsection{Special $\mathbb{RP}^2$-annuli and tori} 

This subsection introduces basic $\mathbb{RP}^2$-surfaces used for grafting.

 A \textit{special $\mathbb{RP}^2$-annulus} is a compact $\mathbb{RP}^2$-manifold with geodesic boundary homotopy-equivalent to a circle. Now we consider the special $\mathbb{RP}^2$-annulus whose holonomy is (\textit{positive}) \textit{hyperbolic}, i.e., it has three distinct positive eigenvalues. Let $A \in PGL(3,\mathbb{R})$ be a hyperbolic transformation. $A$ has three fixed points $p_+,p_-,p_o$. The fixed point $p_+$ corresponding to the largest eigenvalue of $A$ is called an \textit{attracting point}; similarly $p_-$ corresponding to the smallest eigenvalue and $p_o$ corresponding to the middle eigenvalue are called a \textit{repelling point} and a \textit {saddle point} respectively. The transformation $A$ preserves exactly three projective lines passing two of the fixed points $\overline{p_o p_+}$, $\overline{p_+ p_-}$ and $\overline{p_- p_o}$, and as well as triangular domains bounded by them. 

 An \textit{elementary annulus} is a special $\mathbb{RP}^2$-annulus whose interior is the quotient of an invariant triangular domain of the hyperbolic transformation $A$. Elementary annulus is fundamental in the following sense;

\begin{annulus decomposition theorem}
 Let $M$ be a special $\mathbb{RP}^2$-annulus. Then $M$ decomposes into elementary annuli.
\end{annulus decomposition theorem}

 Goldman proved this theorem and classified all special $\mathbb{RP}^2$-annuli in \cite[Corollary 3.6]{Goldman:fuc}. It was also presented in a more general setting in Choi's paper, see \cite[Appendix B]{Choi:admII}.

 A \textit{special $\mathbb{RP}^2$-torus} is a closed orientable $\mathbb{RP}^2$-manifold $T$ homeomorphic to a torus which possesses a simple closed geodesic $C$ whose holonomy is a hyperbolic transformation $A$ and such that the holonomy representation of $T$ surjects onto $\langle A \rangle$. 
 
 We obtain a special $\mathbb{RP}^2$-annulus $T|C$ by cutting $T$ along $C$. $T|C$ decomposes into elementary annuli: $T|C=E_1 \cup E_2 \cup \cdots \cup E_n$. Each elementary annulus $E_i$ is bounded by two simple closed geodesic, one developing to the principal line $\overline{p_+ p_-}$ and the other developing to $\overline{p_- p_o}$ or $\overline{p_o p_+}$. We denote the simple closed geodesics in $T|C$ which develop to $\overline{p_o p_+}$, $\overline{p_+ p_-}$ and  $\overline{p_- p_o}$ by $V_1$,$V_2$ and $V_3$, respectively. We may assume $C$ develops to the principal line without loss of generality. By ordering the simple closed geodesics from one boundary component of $T|C$ to the other, we obtain a finite sequence of $V_i$ ($i=1,2,3$). This sequence is subject to two constraints: (i) no $V_i$ may follow itself; (ii) $V_1$ and $V_3$ may never follow each other. Furthermore this sequence begins and ends with $V_2$.  The conditions (i) and (ii) imply that this sequence takes the form $(V_2 V_{i_1}) \cdots (V_2 V_{i_n})V_2$ where the $i_1, \cdots, i_n$ are either 1 or 3. Hence we may rewrite this sequence as a positive word $w(x,y)$ where $x$ represents $V_2V_1$ and $y$ represents $V_2V_3$. This word $w(x,y)$ consists of an even number of $x$ and an even number of $y$ because the two boundary components develop to the same principal segment.

 We say that a word $w(x,y)$ is \textit{completely even} if it consists of an even number of $x$ and an even number of $y$.

\begin{definition} \label{def:2.1}
  
A completely even word $w$ is \textit{primitive} if $w$ is non-trivial and it cannot be written as $w=(w')^n$ for some $n \ge 2$ and some completely even word $w'$. A completely even word $w$ is \textit{irreducible} if $w=w_1w_2$ implies $w_1=1$ or $w_2=1$.

Let ${}^t w$ denote the word obtained by reading $w$ from right to left and $\overline{w}$ denote the word obtained by interchanging $x$ and $y$.
We define $w^*$ to be $\overline{{}^t w} = {}^t {\overline{w}}$.
\end{definition}

 The set of isomorphism classes of special $\mathbb{RP}^2$-tori corresponds bijectively to the set of all cyclic equivalence classes of completely even words. We denote the special $\mathbb{RP}^2$-torus corresponding to a completely even word $w$ by $T^{A}_{w}$.

\subsection{Grafting torus}

 We focus on a simple closed curve $\alpha$ on a $\mathbb{RP}^2$-surface $M$ along which grafting can be performed. Suppose $\alpha$ has a hyperbolic holonomy $A$. Requiring that the holonomy of the original surface $M$ remain unchanged, we are led to consider a $\mathbb{RP}^2$-torus $T$ satisfying the following conditions:

 \begin{enumerate}
   \item[(i)] the curve $\alpha$ embeds into $T$,
   \item[(ii)] there exists a simple closed curve $\beta \subset T$ that intersects $\alpha$ exactly once and has trivial holonomy.
 \end{enumerate}

 The condition (ii) is equivalent to the holonomy representation of $T$ being surjective onto $\langle A \rangle$.

\begin{definition} \label{def:2.2}
 A \textit{grafting torus} is a $\mathbb{RP}^2$-torus whose holonomy representation is a surjection onto $\langle A \rangle$ generated by a hyperbolic transformation $A$.
\end{definition}

 A special $\mathbb{RP}^2$-torus is a grafting torus. A grafting torus has a simple closed curve whose holonomy is $A$. 

\begin{proposition} \label{prop:2.3}
 A grafting torus can be obtained as follows: let $m$ and $n$ are coprime integers, including the pair $(m,n)=(1,0)$. First, split a special $\mathbb{RP}^2$-torus $T$ which possesses a simple closed geodesic $C$ whose holonomy is $A^m$. Then re-glue $T|C$ via a projective automorphism $A^n$ of the tubular neighborhood of $C$.
\end{proposition}

\begin{proof}
 Let $T$ be a torus and $p \colon \widetilde{T} \to T$ be an universal covering. Suppose $dev \colon \widetilde{T} \to \mathbb{RP}^2$ be a devloping map corresponding to a real projective structure on $T$ whose holonomy is a surjection onto $\langle A \rangle$. Assume $dev(\widetilde{T})$ does not intersect with the lines $l_1,l_2,l_3 (l=l_1 \cup l_2 \cup l_3)$ on $\mathbb{RP}^2$ preserved by $A$. We choose a Riemannian metric $g$ on $\mathbb{RP}^2 \setminus l$ preserved by $A$ for example, $g = \frac{{dx}^2}{x^2}+\frac{{dy}^2}{y^2}$ in an affine chart coordinate. Then we get a Riemannian metric $g'$ on $T$ such that $p^{*}g'=dev^{*}g$. $g'$ is complete because $T$ is compact, and hence $dev^{*}g$ is complete. Therfore $dev \colon \widetilde{T} \to \mathbb{RP}^2$ is a diffeomorphism onto a triangle domain bounded by $l_1,l_2,l_3$. This implies $\pi_{1}(T) \cong \mathbb{Z}$. This is a contradiction.

 We assume $dev(\widetilde{T})$ intersects with $l_i$. ${dev}^{-1}(l_i)$ is preserved by the deck transformations. Therefore $p({dev}^{-1}(l_i))$ is a compact submanifold in $T$. Let $C$ be a connected component of $p({dev}^{-1}(l_i))$. $C$ is a simple closed curve which develops to the line $l_i$. $T|C$ is a hyperbolic annulus with convex boundaries. The annulus decompostion theorem implies that $T|C$ decomposes into elementary annuli.

 The following holds because of the holonomy $\pi_{1}(T) \twoheadrightarrow \langle A \rangle$;
\begin{itemize}
  \item the holonomy of the geodesic $C$ is $A^m (m \in \mathbb{Z} \setminus \{0\})$,
  \item the real projective structure on T is obtained by gluing $T|C$ along $C$ by a projective isomorphism $A^n$,
  \item $m$ and $n$ are coprime.
\end{itemize}

\end{proof}

\subsection{Graftable curves}
 Let $M$ be a $\mathbb{RP}^2$-surface and $A \in PGL(3,\mathbb{R})$ be a (positive) hyperbolic transformation. The \textit{holonomy covering space} $\widehat{T}$ of a special $\mathbb{RP}^2$-torus T is the quotient of the universal covering space $\widetilde{T}$ by the kernel of the holonomy representation. The developing map $\widetilde{T} \to \mathbb{RP}^2$ of $T$ descends to the map $\widehat{T} \to \mathbb{RP}^2$.

\begin{definition} \label{def:2.4}
 A simple closed curve $\alpha$ on $M$ is \textit{graftable} if the holonomy $A$ of $\alpha$ is hyperbolic and there is a special $\mathbb{RP}^2$-torus $T$ with the holonomy $\pi_{1}(T) \twoheadrightarrow \langle A \rangle$ such that, letting $\widetilde{\alpha} \subset \widetilde{M}$ be a lift of $\alpha$, 

(i) there is a lift $\widetilde{dev} \colon \widetilde{\alpha} \to \widehat{T}$ of the developing map restricted to $\widetilde{\alpha}$, and

(ii) this lift is injective (embedding).

\begin{figure}[h]
  \centering
  \includegraphics[width=4cm]{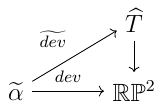}
\end{figure}
\end{definition}

 A \textit{graftble multi-curve} is a multi-curve consisting of disjoint graftable simple closed curves.

\subsection{Grafting along closed curves}

\begin{proposition} \label{prop:2.5}
 Let $\alpha$ be a graftable simple closed curve on a $\mathbb{RP}^2$-surface $M$. Then $\alpha$ can be projectively embedded in a grafting torus.
\end{proposition}

\begin{proof}
 There is a special $\mathbb{RP}^2$-torus $T$ satisfying two conditions (i),(ii) in the definition. Let $\theta \in \pi_1(M)$ be the deck transformation corresponding to $\alpha$. We have a projective automorphism $f$ of the holonomy covering space $\widehat{T}$ such that the following diagram

\begin{figure}[H]
    \centering
    
    \begin{minipage}[h]{0.24\textwidth}
        \centering
        \includegraphics[width=3cm]{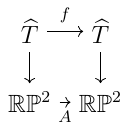}

    \end{minipage}
    \begin{minipage}[h]{0.24\textwidth}
        \centering
        \includegraphics[width=3cm]{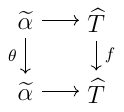}
    \end{minipage}
    
\end{figure}

commutes.
$\langle f \rangle$ acts freely and properly discontinuously on $\widehat{T}$ and the simple closed curve $\alpha$ can be projectively embedded in a grafting torus $\widehat{T} /{\langle f \rangle}$.
\end{proof}

\begin{remark} \label{rem:2.6}
 Let $\alpha$ be a graftable simple closed curve. There are many grafting tori which $\alpha$ can be projectively embedded in.  
\end{remark}

 If $\alpha$ is a graftable simple closed curve, one can produce another projective structure without changing the holonomy by grafting along $\alpha$. The construction is the followings (compared to \cite[\S 2.2]{CDF:gra} and also compared to \cite[\S 1 and \S 4.2]{Goldman:fuc}). We cut the universal cover $\widetilde{\Sigma}$ along all lifts $\widetilde{\alpha}$ of the simple closed curve $\alpha$. Along each such lift, we insert a copy of the projective strip $\widehat{T} \setminus \widetilde{dev}(\widetilde{\alpha})$. The gluing is carried out using the lift of the developing map. We then obtain a new universal cover $\widetilde{\Sigma}'$, together with a new developing map $D' \colon \widetilde{\Sigma}' \to \mathbb{RP}^2$, which is defined by $D$ on $\widetilde{\Sigma}$ and by the identity on $\widehat{T} \setminus \widetilde{dev}(\widetilde{\alpha})$. The space $\widetilde{\Sigma}'$ carries a natural action of $\pi_1(\Sigma)$ induced from the original one. The action is free and properly discontinuous, and the extended developing map $D'$ is $\rho$-equivariant. Hence, this defines a new real projective structure on $\Sigma$. 
 
 By Proposition \ref{prop:2.5}, there is a grafting torus $\widehat{T} / \langle f \rangle$ which $\alpha$ can be projectively embedded in. Therefore the grafting can be viewed as a cut-and-paste procedure directly in $\Sigma$, which cuts $\Sigma$ along $\alpha$ and glues back the cylinder $(\widehat{T} / \langle f \rangle)|\alpha$.

 We denote this new real projective structure on $\Sigma$ by $\mathbf{Gr}(\sigma,(\alpha,\widehat{T} / \langle f \rangle))$.

\subsection{Isotopy of graftable curves}

\begin{definition} \label{def:2.7}
  Let $\alpha$ and ${\alpha}'$ be graftable simple closed curves on a real projective surface $M$.

  $\alpha$ and ${\alpha}'$ are \textit{isotopic} (as graftable curves) if there is a special $\mathbb{RP}^2$-torus $T$ and an isotopy $F \colon S^1 \times [0,1] \to M$ such that

(i) there is a lift of $dev \circ \widetilde{F}$,

(ii) $\overline{F_t} \colon \mathbb{R} \times \{t\} \to \widehat{T}$ is injective at each level $t \in [0,1]$.

\begin{figure}[h]
  \centering
  \includegraphics[width=8cm]{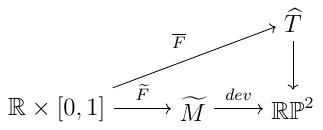}
\end{figure}

\end{definition}

 If two graftable simple closed curves $\alpha$ and ${\alpha}'$ are isotopic, then  $\alpha$ and ${\alpha}'$ are both projectively embedded in the same grafting torus $T$. Then, the resulting real projective structures $\mathbf{Gr}(M,(\alpha, T))$ and $\mathbf{Gr}(M,({\alpha}', T))$ are equivarent.

\section{Construction of graftable curves}

\subsection{Convex geometry}

 In this subsection, we review basic notions and results on convex real projective geometry, and prepare the notation that will be used for in the latter subsection.

 A domain $\Omega \subset \mathbb{RP}^2$ is \textit{convex} if it is contained in some affine chart and convex in the usual sense. It is called \textit{properly convex} if its closure $\overline{\Omega}$ is contained in the affine chart. Furthermore, $\Omega$ is \textit{strictly convex} if it is properly convex and its boundary $\partial \Omega$ contains no line segment.

 Let $\Sigma$ be a closed oriented surface with genus $> 1$. A \textit{convex real projective structure} on $\Sigma$ is a real projective structure whose developing map $dev \colon \widetilde{\Sigma} \to \mathbb{RP}^2$ is a diffeomorphism onto a convex domain $\Omega$ in $\mathbb{RP}^2$. The holonomy representation $\rho \colon \pi_1(\Sigma) \to PGL(3,\mathbb{R})$ is a \textit{Hitchin representation}, i.e., it can be deformed to a faithful and discrete representation. Conversely, Goldman and Choi \cite{CG:hit} proved that, given a Hitchin representation $\rho \colon \pi_1(\Sigma) \to PGL(3,\mathbb{R})$, there is a unique convex real projective structure whose holonomy representation $\rho$. In this case, $\Omega$ is strictly convex and its boundary $\partial \Omega$ is class of $C^1$ (see \cite[Theorem 3.2 (1)]{Goldman:con}). The universal cover $\widetilde{\Sigma}$ is identified with $\Omega$, and the surface $\Sigma$ is realized as the quotient $\Omega / \rho(\pi_1(\Sigma))$.

 Each non-trivial element $A$ of the holonomy group $\rho(\pi_1(\Sigma))$ is positive hyperbolic (see \cite[Lemma 4.3]{Choi:admII} and \cite[Theorem 3.2 (3)]{Goldman:con}). The attracting point $p_{+}$ and the repelling point $p_{-}$ lie on the boundary $\partial \Omega$, while the saddle point $p_{o}$ lies outside of $\Omega$. The principal line $\overline{p_{+}p_{-}}$ separates $\Omega$ into two regions $\Omega_{+}$ and $\Omega_{-}$, determined by the orientation of the principal segment in $\Omega$. There are four triangular domains $\triangle^{A}_{k}$ ($k=1,2,3,4$) preserved by $A$ so that $\Omega_{+} \subset \triangle^{A}_{1}$, $\Omega_{-} \subset \triangle^{A}_{4}$, $\overline{\triangle^{A}_{1}} \cap \overline{\triangle^{A}_{2}} \subset \overline{p_{o}p_{+}}$, $\overline{\triangle^{A}_{3}} \cap \overline{\triangle^{A}_{1}} \subset \overline{p_{-}p_{o}}$; see Figure \ref{fig:1}. We denote by $l_{i}^{j}$ the (open) line segment in $\overline{\triangle^{A}_{i}} \cap \overline{\triangle^{A}_{j}}$ connecting two of the fixed points. We define the hexagonal domains $H^{A}_{\pm *}$ consisting of the copies of $\triangle^{A}_k$ and $l_{i}^{j}$:  \[H^{A}_{+xx} = l_{4}^{1} \cup \triangle^{A}_1 \cup l_{1}^{2} \cup \triangle^{A}_2 \cup l_{2}^{3} \cup \triangle^{A}_3 \cup l_{3}^{4} \cup \triangle^{A}_4,\]
 \[H^{A}_{-xx} = l_{1}^{4} \cup \triangle^{A}_4 \cup l_{4}^{3} \cup \triangle^{A}_3 \cup l_{3}^{2} \cup \triangle^{A}_2 \cup l_{2}^{1} \cup \triangle^{A}_1,\]
 \[H^{A}_{+xy} = l_{4}^{1} \cup \triangle^{A}_1 \cup l_{1}^{2} \cup \triangle^{A}_2 \cup l_{2}^{3} \cup \triangle^{A}_3 \cup l_{3}^{1} \cup \triangle^{A}_1,\]
 \[H^{A}_{-xy} = l_{1}^{4} \cup \triangle^{A}_4 \cup l_{4}^{3} \cup \triangle^{A}_3 \cup l_{3}^{2} \cup \triangle^{A}_2 \cup l_{2}^{4} \cup \triangle^{A}_4,\]
 \[H^{A}_{+yy} = l_{4}^{1} \cup \triangle^{A}_1 \cup l_{1}^{3} \cup \triangle^{A}_3 \cup l_{3}^{2} \cup \triangle^{A}_2 \cup l_{2}^{4} \cup \triangle^{A}_4,\]
 \[H^{A}_{-yy} = l_{1}^{4} \cup \triangle^{A}_4 \cup l_{4}^{2} \cup \triangle^{A}_2 \cup l_{2}^{3} \cup \triangle^{A}_3 \cup l_{3}^{1} \cup \triangle^{A}_1,\]
 \[H^{A}_{+yx} = l_{4}^{1} \cup \triangle^{A}_1 \cup l_{1}^{3} \cup \triangle^{A}_3 \cup l_{3}^{2} \cup \triangle^{A}_2 \cup l_{2}^{1} \cup \triangle^{A}_1,\]
 \[H^{A}_{-yx} = l_{1}^{4} \cup \triangle^{A}_4 \cup l_{4}^{2} \cup \triangle^{A}_2 \cup l_{2}^{3} \cup \triangle^{A}_3 \cup l_{3}^{4} \cup \triangle^{A}_4.\]
 
 A strictly convex domain $\Omega$ carries a natural metric, called the \textit{Hilbert metric}. For two distinct points $x,y \in \Omega$, let $l$ be the line through $x$ and $y$, and let $a,b \in \partial \Omega$ be the intersection points of $l$ with the boundary, ordered so that $a,x,y,b$ lie on $l$ in this order. The Hilbert distance $d_{H}$ between $x$ and $y$ is defined by 
 \[d_{H}(x,y) = \frac{1}{2} [a,x,y,b],\] 
where $[a,x,y,b]$ denotes the closs ratio of the four collinear points. Since the projective automorphisms of $\Omega$ preserve $d_{H}$, the Hilbert metric on $\Omega$ descends to the metric on the convex real projective surface $\Sigma = \Omega / \rho(\pi_1(\Sigma))$. The curve $\alpha$ on $\Sigma$ is a geodesic if and only if each lift of $\alpha$ to the universal cover $\Omega$ is a line segment. Every non-trivial free homotopy class is represented by a unique closed geodesic. Furthermore, if the closed geodesic $\alpha$ is homotopic to a simple closed curve, then $\alpha$ is simple. Any two simple closed geodesics (or multi-geodesics) are in minimal position, realizing the minimal possible number of intersections in their free homotopy classes. (See \cite[\S 3.5]{Goldman:con})

 \begin{figure}[H]
  \begin{overpic}[width=6.5cm,clip]{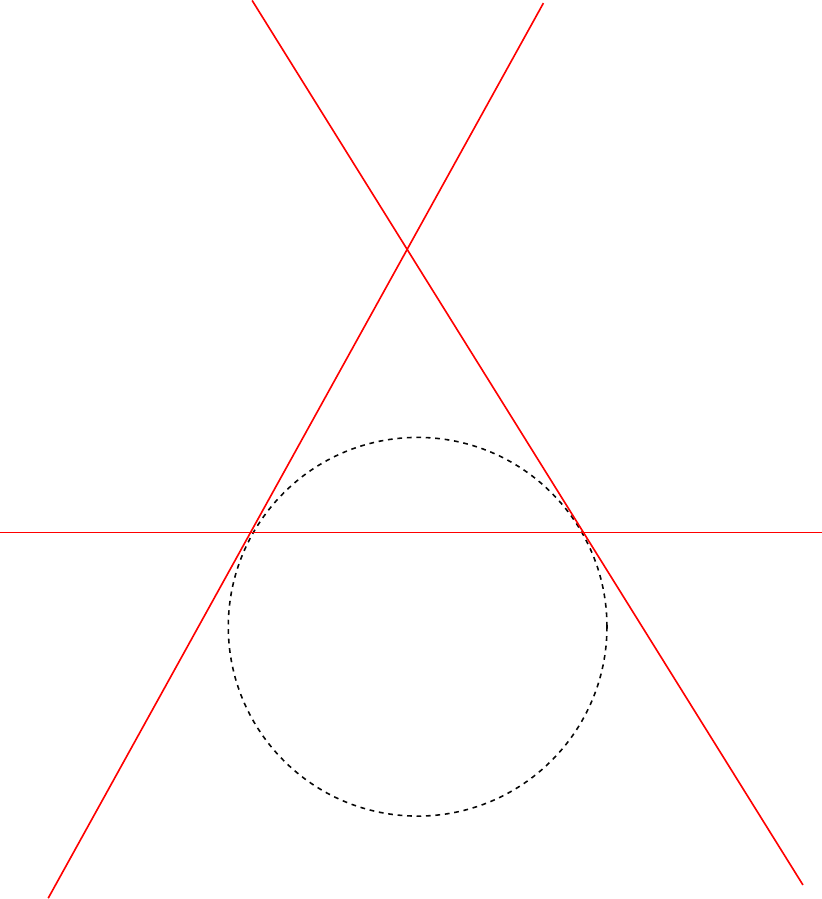}
  
  \put (45,85) {$p_{+}$}
  \put (132,85) {$p_{-}$}
  \put (95,145) {$p_{o}$}
  
  \put (87,87) {$\Omega_{+}$}
  \put (87,50) {$\Omega_{-}$}
  
  \put (86,110) {$\triangle^{A}_{1}$}
  \put (30,135) {$\triangle^{A}_{2}$}
  \put (165,50) {$\triangle^{A}_{2}$}
  \put (140,135) {$\triangle^{A}_{3}$}
  \put (10,50) {$\triangle^{A}_{3}$}
  \put (86,180) {$\triangle^{A}_{4}$}
  \put (86,25) {$\triangle^{A}_{4}$}
  
  \end{overpic}

  \caption{The triangular domains and lines preserved by a hyperbolic transformation $A$.}
  \label{fig:1}
\end{figure}
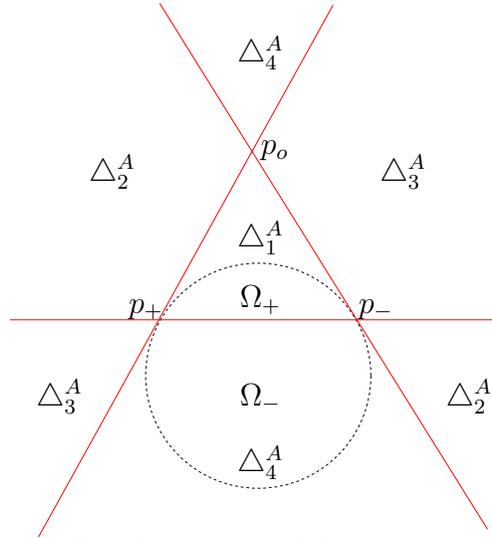

\subsection{Real projective strctures with Hitchin holonomy}
 We denote by $\sigma_c$ the convex real projective structure on the surface $\Sigma_{c} = \Omega/\rho(\pi_1(\Sigma))$ corresponding to a Hitchin representation $\rho \colon \pi_1(\Sigma) \to PGL(3,\mathbb{R})$. The developing map of this structure $\sigma_c$ is the diffeomorphism onto a convex domain $\Omega \subset \mathbb{RP}^2$. Hence, any simple closed curve $\alpha$, which has a hyperbolic holonomy $A \in PGL(3,\mathbb{R})$, is graftable and may be projectively embedded in any special $\mathbb{RP}^2$-torus $T^{A}_{w}$. Furthermore, the grafting $\mathbf{Gr}(\sigma_c, (\alpha, T^{A}_{w}))$ depends only on the isotopy class of $\alpha$ as just a curve.

 \begin{theorem}[Goldman \cite{Goldman:fuc}]
  Every real projective structure $\sigma$ with Hitchin holonomy $\rho$ is obtained by grafting the convex structure $\sigma_c$ along a multi curve $\alpha = \{\alpha_i\}$.
 \end{theorem}
 This family of pairs of simple closed curves and irreducible completely even words can be reconstructed from $\sigma$ as follows: let $\widetilde{\mathbb{RP}^2}$ be the double cover of $\mathbb{RP}^2$. Since $\Sigma$ is orientable, the real projective structure $\sigma$ can be lifted to the ($\widetilde{\mathbb{RP}^2}$ , $SL(3,\mathbb{R})$)-structure. The flat $\widetilde{\mathbb{RP}^2}$-bundle $\mathcal{E}_{\sigma}$ over $\Sigma$ decomposes into two subbundles the $\Omega$-bundle $\mathcal{E}^{+}_{\sigma}$ and the $(\widetilde{\mathbb{RP}^2} \setminus \Omega)$-bundle $\mathcal{E}^{-}_{\sigma}$ since the holonomy group preserves the decomposition $\mathbb{RP}^2 = \Omega \cup (\widetilde{\mathbb{RP}^2} \setminus \Omega)$. The developing section $\mathcal{D}_{\sigma} \colon \Sigma \to \mathcal{E}_{\sigma}$ induces a decomposition of $\Sigma$ into subsurfaces ${\Sigma}^{+}={\mathcal{D}_{\sigma}}^{-1}(\mathcal{E}^{+}_{\sigma})$ and ${\Sigma}^{-} = {\mathcal{D}_{\sigma}}^{-1}(\mathcal{E}^{-}_{\sigma})$. Goldman proved that each component of ${\Sigma}^{-}$ need to be an annulus by computing the Euler class of the flat bundle $\mathcal{E}_{\sigma}$.
 
 Now consider a disjoint collection ${\alpha}_i$ of homotopically nontrivial simple closed curves, that is, an isotopy class of a multi-curve. For each $i$ choose an irreducible completely even word $w_i$. We call such a collection $\eta = (\alpha_1, \cdots, \alpha_n ,w_1, \cdots, w_n)$ a \textit{grafting data} and denote by $\mathscr{GD}$ the set of grafting data. Two grafting data $({\alpha}_i,w_i)$ and $({{\alpha}_i}^{-1},{w_i}^{*})$ are equivalent. The real projective structures with Hitchin holonomy $\rho$ corresponds bijectively to the grafting data $\mathscr{GD}$. We denote by $\mathbf{Gr}_{\eta}$ the real projective structure with holonomy $\rho$ corresponding to a grafting data $\eta \in \mathscr{GD}$. 

\subsection{Construction of graftable multi-curves}
 Given a grafting data $\eta = (\{\alpha_i\}_{i \in I},\{w_i\}_{i \in I})$ for a finite index set $I$. In this section, we provide a construction that, given two homotopically transverse multicurves $\alpha = \{\alpha_i\}_{i \in I}$ and $\beta = \{\beta_j\}_{j \in J}$, makes a graftable multi-curve $\beta_{R},\beta_{L}$ in $\mathbf{Gr}_{\eta}$ isotopic to $\beta$.

 We will mainly explain the construction of $\beta_{R}$; the construction of $\beta_L$ is similar to that of $\beta_{R}$ and will be discussed at the end. If $\alpha_{i_1}$ and $\alpha_{i_2}$ are parallel, then we can multiply $w_{i_1}$ and $w_{i_2}$. In this way we may assume $\alpha$ contains no parallel loops. We begin by assuming that the multi-curve $\beta$ also contains no parallel loops. In this case, one can assume that each component of $\alpha$ and $\beta$ are simple closed geodesics. Recall that $\mathbf{Gr}_{\eta}$ is obtained by gluing ${\Sigma}_c|\alpha$ with some grafting annuli.

 The boundary of $\Sigma_c|\alpha$ consists of two copies ${\alpha}_i'$ and ${\alpha}_i''$ of each curve ${\alpha}_i$ and there are two copies $p' \in \alpha_i'$ and $p'' \in \alpha_i''$ of each point $p \in \alpha_i \cap \beta$. We fix a small positive number $\epsilon$, consider the points $p_{R}' \in {\alpha_i}'$ and $p_{R}'' \in {\alpha_i}''$ lying at distance $\epsilon$ (with respect to the Hilbert metric) from $p'$ and $p''$ to the right side of $p$ respectively.

 Now, $\beta \cap (\Sigma_c|\alpha)$ is a union of geodesic segments $[p,q]$ joining points of the boundary of $\Sigma_c|\alpha$. We define $\beta_{R}$ in $\Sigma_c|\alpha \subset \Sigma$ to be the union of the segments $[p_{R},q_{R}]$ with $p_{R}$ and $q_{R}$ constructed above. Observe that the segments $[p_{R},q_{R}]$ are disjoint each other if $\epsilon$ is small enough.

 Then, we define the multi-curve $\beta_{R}$ in the grafting annuli. First, we suppose that $\beta$ intersects $\alpha_i$ once. We may suppose $\beta$ and $\alpha_i$ intersect compatibly with the orientation of $\Sigma$ without loss of generality. The grafting annulus inserted along $\alpha_i$ corresponds to the completely even word $w_i$ and decomposes into eight types of special $\mathbb{RP}^2$-annuli $F^{\alpha_i}_{\pm xx},F^{\alpha_i}_{\pm xy},F^{\alpha_i}_{\pm yy}$ and $F^{\alpha_i}_{\pm yx}$: Let $A_i = \rho(\alpha_i) \in PGL(3,\mathbb{R})$. Each annuli $F^{\alpha_i}_{\pm *}$ decomoses into four elementary annuli and corresponds to the index word $*$. The induced developing map is onto a hexagonal domain $H^{A_i}_{\pm *}$, as described in Subsection 3.1. 

 In each annulus $F^{\alpha_i}_{\pm *}$, construct a simple arc as shown in the following Figure\ref{fig:2} (Figure \ref{fig:3}).
 
 \begin{figure}[p]
   \centering
   \begin{overpic}[width=16cm,clip]{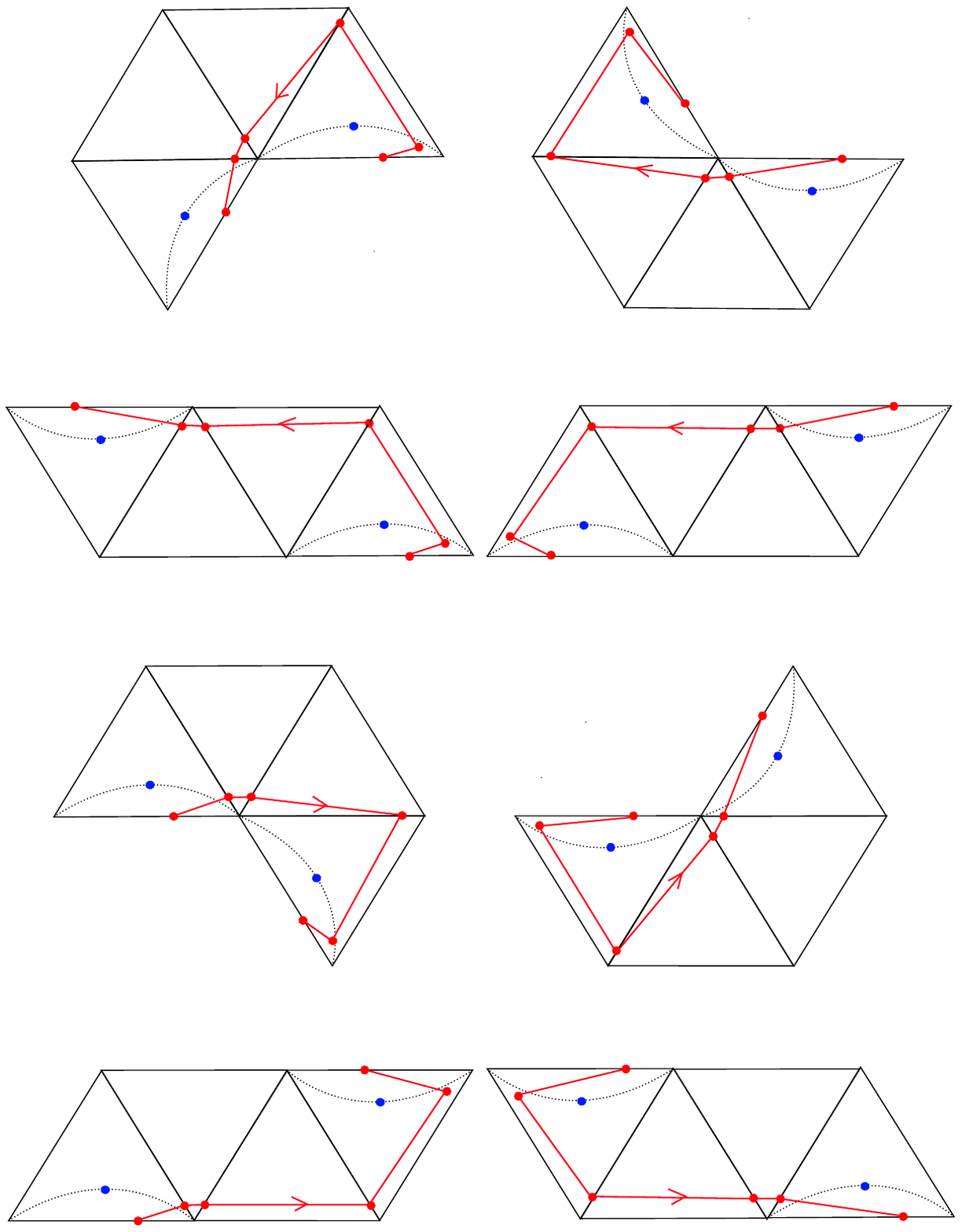}

   \put (100,460) {$H^{A_i}_{+xx}$}
    \put (160,580) {$\triangle^{A_i}_1$}
    \put (110,595) {$\triangle^{A_i}_2$}
    \put (70,570) {$\triangle^{A_i}_3$}
    \put (60,530) {$\triangle^{A_i}_4$}
    
   \put (330,460) {$H^{A_i}_{-xx}$}
    \put (375,510) {$\triangle^{A_i}_4$}
    \put (330,500) {$\triangle^{A_i}_3$}
    \put (290,520) {$\triangle^{A_i}_2$}
    \put (280,570) {$\triangle^{A_i}_1$}
    
   \put (100,340) {$H^{A_i}_{+xy}$}
    \put (175,390) {$\triangle^{A_i}_1$}
    \put (130,400) {$\triangle^{A_i}_2$}
    \put (85,380) {$\triangle^{A_i}_3$}
    \put (40,400) {$\triangle^{A_i}_1$}
    
   \put (330,340) {$H^{A_i}_{-xy}$}
    \put (400,400) {$\triangle^{A_i}_4$}
    \put (355,385) {$\triangle^{A_i}_3$}
    \put (310,400) {$\triangle^{A_i}_2$}
    \put (270,390) {$\triangle^{A_i}_4$}
    
   \put (100,150) {$H^{A_i}_{+yy}$}
    \put (65,265) {$\triangle^{A_i}_1$}
    \put (105,280) {$\triangle^{A_i}_3$}
    \put (150,260) {$\triangle^{A_i}_2$}
    \put (155,215) {$\triangle^{A_i}_4$}
    
   \put (330,150) {$H^{A_i}_{-yy}$}
    \put (285,205) {$\triangle^{A_i}_4$}
    \put (325,190) {$\triangle^{A_i}_2$}
    \put (370,210) {$\triangle^{A_i}_3$}
    \put (370,255) {$\triangle^{A_i}_1$}
    
   \put (100,20) {$H^{A_i}_{+yx}$}
    \put (40,75) {$\triangle^{A_i}_1$}
    \put (85,90) {$\triangle^{A_i}_3$}
    \put (130,70) {$\triangle^{A_i}_2$}
    \put (175,85) {$\triangle^{A_i}_1$}
    
   \put (330,20) {$H^{A_i}_{-yx}$}
    \put (270,85) {$\triangle^{A_i}_4$}
    \put (310,75) {$\triangle^{A_i}_2$}
    \put (355,90) {$\triangle^{A_i}_3$}
    \put (400,75) {$\triangle^{A_i}_4$}
    
   \end{overpic}
   \caption{$\beta_R$ in the grafting annuli (red lines). Blue points are fixed points of $\rho(\beta_j)$.}
   \label{fig:2}
 \end{figure}

 \begin{figure}[p]
   \centering
   \begin{overpic}[width=16cm,clip]{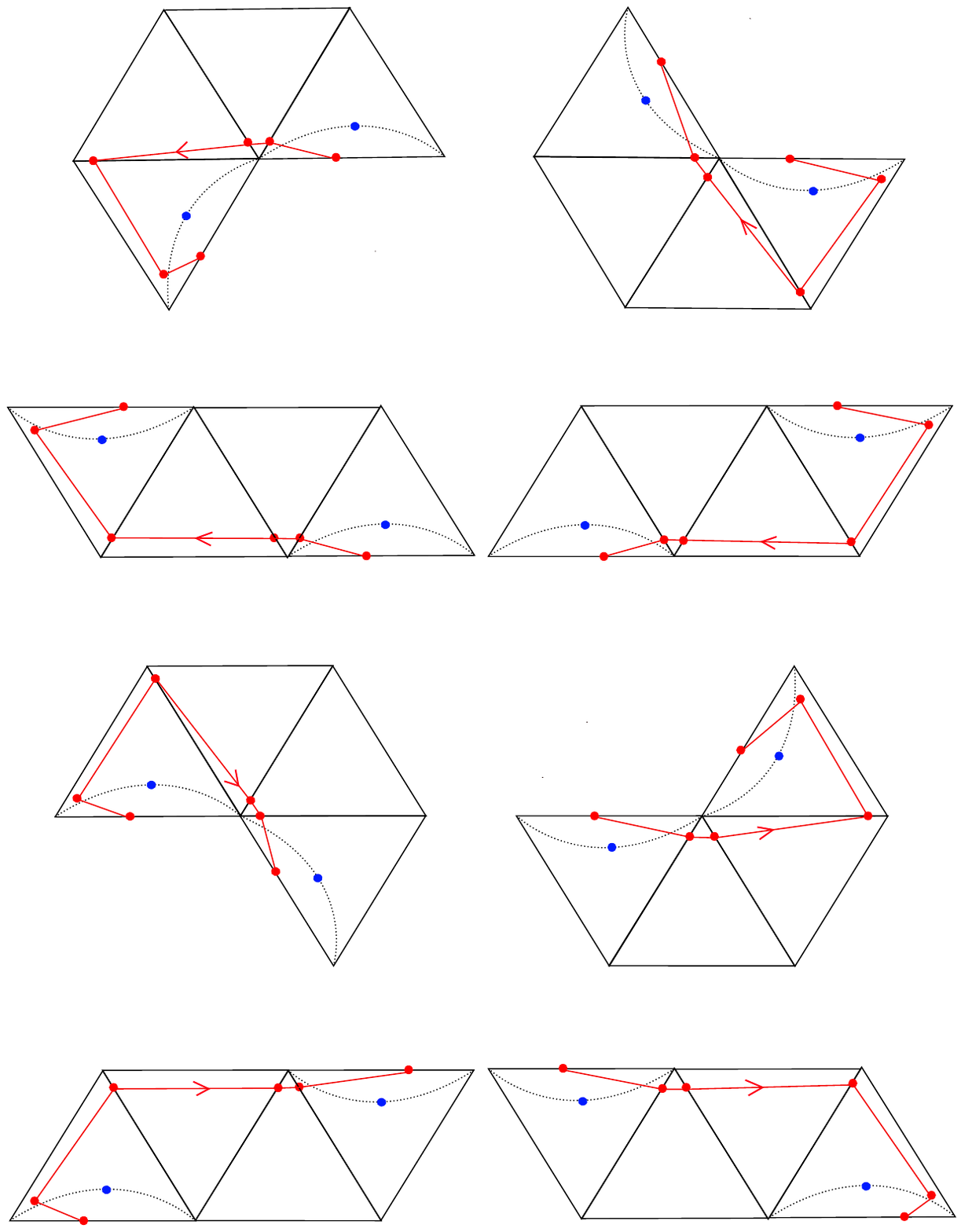}

   \put (100,450) {$H^{A_i}_{+xx}$}
    \put (160,570) {$\triangle^{A_i}_1$}
    \put (115,585) {$\triangle^{A_i}_2$}
    \put (75,565) {$\triangle^{A_i}_3$}
    \put (70,520) {$\triangle^{A_i}_4$}
    
   \put (330,450) {$H^{A_i}_{-xx}$}
    \put (375,510) {$\triangle^{A_i}_4$}
    \put (335,495) {$\triangle^{A_i}_3$}
    \put (290,515) {$\triangle^{A_i}_2$}
    \put (285,565) {$\triangle^{A_i}_1$}
    
   \put (100,330) {$H^{A_i}_{+xy}$}
    \put (175,385) {$\triangle^{A_i}_1$}
    \put (130,400) {$\triangle^{A_i}_2$}
    \put (85,380) {$\triangle^{A_i}_3$}
    \put (40,400) {$\triangle^{A_i}_1$}
    
   \put (330,330) {$H^{A_i}_{-xy}$}
    \put (400,400) {$\triangle^{A_i}_4$}
    \put (355,385) {$\triangle^{A_i}_3$}
    \put (315,395) {$\triangle^{A_i}_2$}
    \put (270,385) {$\triangle^{A_i}_4$}
    
   \put (100,140) {$H^{A_i}_{+yy}$}
    \put (65,260) {$\triangle^{A_i}_1$}
    \put (105,280) {$\triangle^{A_i}_3$}
    \put (150,255) {$\triangle^{A_i}_2$}
    \put (155,210) {$\triangle^{A_i}_4$}
    
   \put (330,140) {$H^{A_i}_{-yy}$}
    \put (285,205) {$\triangle^{A_i}_4$}
    \put (325,190) {$\triangle^{A_i}_2$}
    \put (370,210) {$\triangle^{A_i}_3$}
    \put (370,255) {$\triangle^{A_i}_1$}
    
   \put (100,10) {$H^{A_i}_{+yx}$}
    \put (45,70) {$\triangle^{A_i}_1$}
    \put (85,80) {$\triangle^{A_i}_3$}
    \put (130,70) {$\triangle^{A_i}_2$}
    \put (175,85) {$\triangle^{A_i}_1$}
    
   \put (330,10) {$H^{A_i}_{-yx}$}
    \put (270,85) {$\triangle^{A_i}_4$}
    \put (315,70) {$\triangle^{A_i}_2$}
    \put (355,80) {$\triangle^{A_i}_3$}
    \put (400,70) {$\triangle^{A_i}_4$}
    
   \end{overpic}
   \caption{$\beta_L$ in the grafting annuli (red lines). Blue points are fixed points of $\rho(\beta_j)$.}
   \label{fig:3}
 \end{figure}
 
 Then, one obtains $\beta_{R}$ in the grafting annulus which joins $p_{R}'$ and $p_{R}''$ by connecting these arcs.
 
 Now we check that $\beta_{R}$ constructed above is a graftable multi-curve. It suffices to verify that each component $(\beta_j)_R$ of $\beta_R$ is graftable since the components do not intersect one another. More precisely, we will show the following.
First, assign a wight to each intersection point of $\beta$ and $\alpha_i$: if the orientation of $\alpha_i$ and $\beta$ around the point agree with the orientation of the surface $\Sigma$, attach $w_i$; otherwise attach $w_i^{*}$ as in Figure \ref{fig:4}.

 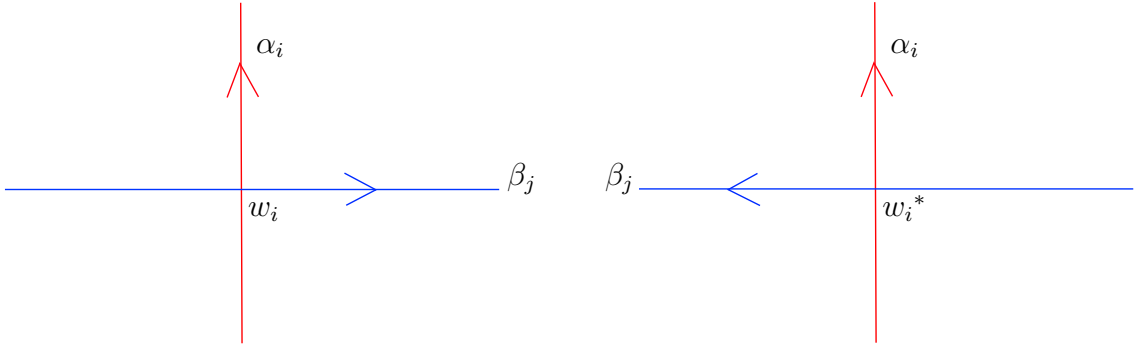
\begin{figure}[H]
  \begin{overpic}[width=15cm,clip]{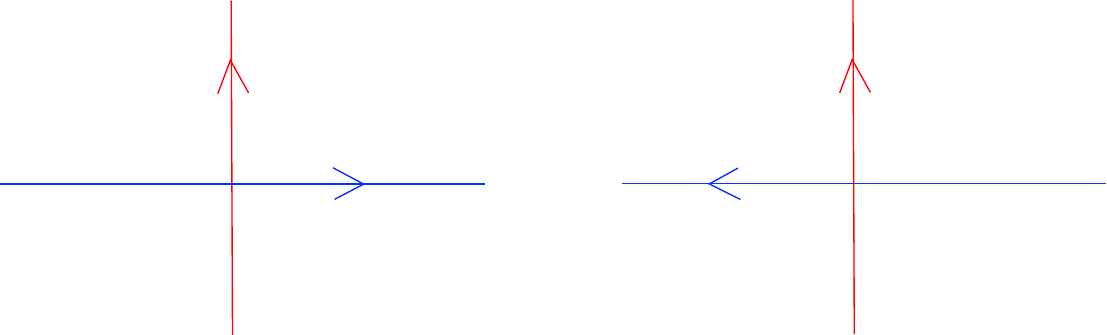}
  \put (95,110) {$\alpha_i$}
  \put (190,60) {$\beta_j$}
  \put (92,48) {$w_i$}

  \put (335,110) {$\alpha_i$}
  \put (227,60) {$\beta_j$}
  \put (332,48) {${w_i}^{*}$}
  \end{overpic}
  \caption{A weighting scheme along $\beta_R$.}
  \label{fig:4}
 \end{figure}

 Then along $\beta_j$, multiply the wights in the order in which the intersection points appear, and denote the resulting completely even word by $\widehat{\tau_j}$. Let $\tau_j$ be the primitive completely even word satisfying $\widehat{\tau_j}={\tau_j}^n$ for some $n \ge 1$. We would like to show
\begin{enumerate}
  \item $(\beta_j)_{R}$ is simple,
  \item there is a lift to $\widehat{T^{\rho(\beta_j)}_{\tau_j}}$ of the developing map restricted to $\widetilde{(\beta_j)_{R}}$.
  \item this lift is injective.
\end{enumerate}

 (1) The grafting annulus consists of elementary annuli. Therefore, we will check that each arc $\gamma$ in the elementary annulus is simple. The lift of $\gamma$ in each triangle $\triangle^{A_i}_{k}$ are formed by connecting line segments in $\mathbb{RP}^2$. Each of these line segments intersects the principal segment when it is extended in $\triangle^{A_i}_{k}$, see Figure \ref{fig:2}. This implies that the lifts of $\gamma$ to the universal cover are disjoint each other and hence $\gamma$ is a simple arc.

 (2) Let $B_j=\rho(\beta_j) \in PGL(3,\mathbb{R})$. The holonomy covering space $\widehat{T^{B_j}_{\tau_j}}$ is a union of the copies of hexagonal domains $H^{B_j}_{\pm *}$. The developing image of the simple arc in $H^{A_i}_{\pm *} $ is shown in Figure \ref{fig:5}. Figure \ref{fig:5} shows that it can be lifted to $H^{B_j}_{\pm *}$.
 
 (3) The triangle domains $\triangle^{A_{i_1}}_1$ and $\triangle^{A_{i_2}}_1$ $(i_1 \neq i_2)$ are disjoint. Likewise, the triangle domains $\triangle^{A_{i_1}}_4$ and $\triangle^{A_{i_2}}_4$ are disjoint as well. Therefore, we can make lifts of simple arcs in $\widehat{T^{B_j}_{\tau_j}}$ mutually disjoint by a homotopy, see Figure \ref{fig:5}. This means the lift of $\widetilde{(\beta_j)_{R}}$ to $\widehat{T^{B_j}_{\tau_j}}$ is injectively embedded.

 \begin{figure}[p]
   \centering
   \begin{overpic}[width=16cm,clip]{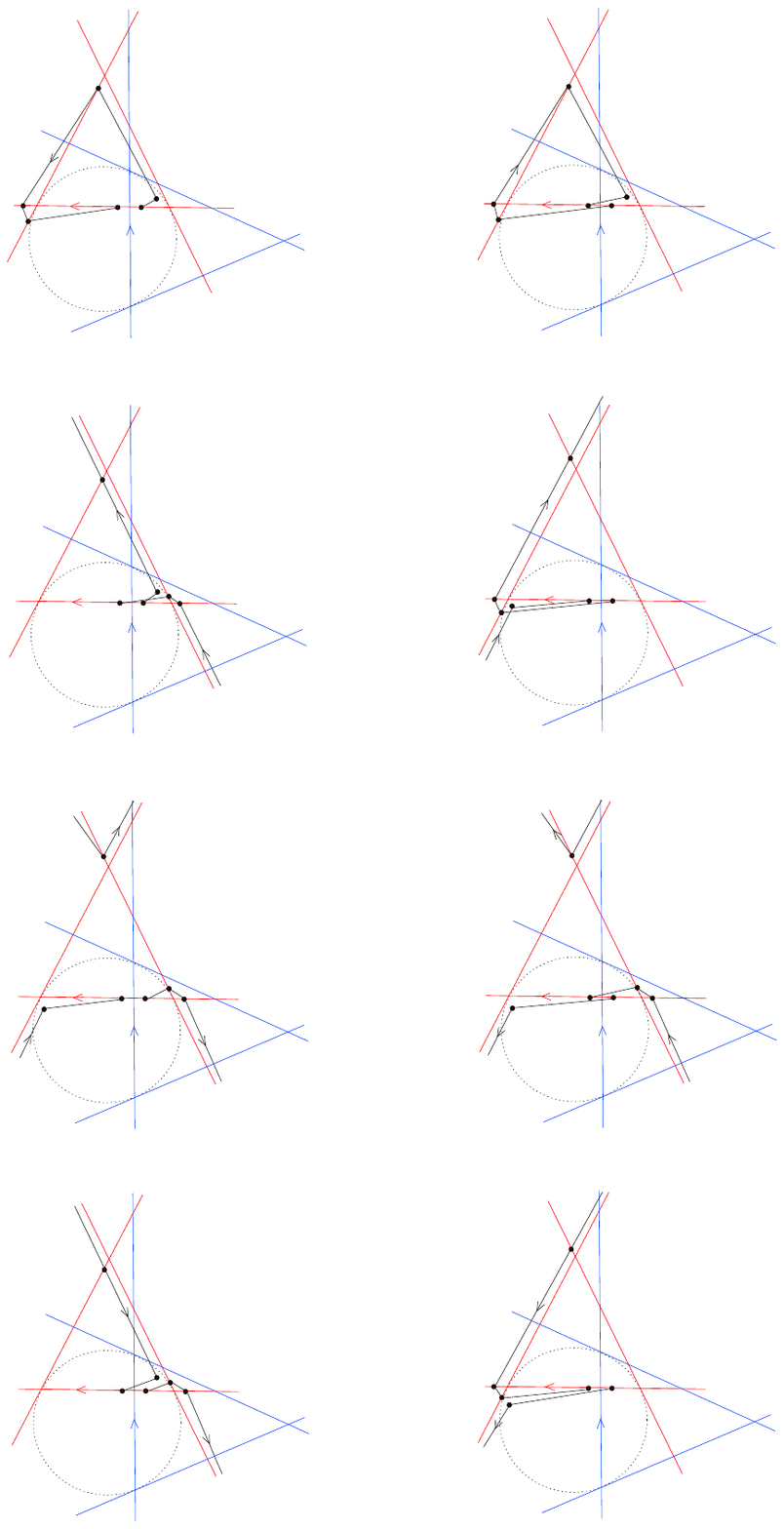}
   
   \put (120,490) {$H^{A_i}_{+xx}$}
    \put (100,562) {$\color{red}{\widetilde{\alpha_i}}$}
    \put (120,520) {$\color{blue}\widetilde{\beta_j}$}
    
   \put (310,490) {$H^{A_i}_{-xx}$}
    \put (290,562) {$\color{red}{\widetilde{\alpha_i}}$}
    \put (310,520) {$\color{blue}{\widetilde{\beta_j}}$}
    
   \put (120,325) {$H^{A_i}_{+xy}$}
    \put (100,402) {$\color{red}{\widetilde{\alpha_i}}$}
    \put (120,360) {$\color{blue}{\widetilde{\beta_j}}$}
    
   \put (310,325) {$H^{A_i}_{-xy}$}
    \put (290,402) {$\color{red}{\widetilde{\alpha_i}}$}
    \put (310,360) {$\color{blue}{\widetilde{\beta_j}}$}
    
   \put (120,167) {$H^{A_i}_{+yy}$}
    \put (100,240) {$\color{red}{\widetilde{\alpha_i}}$}
    \put (120,200) {$\color{blue}{\widetilde{\beta_j}}$}
    
   \put (310,167) {$H^{A_i}_{-yy}$}
    \put (290,240) {$\color{red}{\widetilde{\alpha_i}}$}
    \put (310,200) {$\color{blue}{\widetilde{\beta_j}}$}
    
   \put (120,10) {$H^{A_i}_{+yx}$}
    \put (100,80) {$\color{red}{\widetilde{\alpha_i}}$}
    \put (120,40) {$\color{blue}{\widetilde{\beta_j}}$}
    
   \put (310,10) {$H^{A_i}_{-yx}$}
    \put (290,80) {$\color{red}{\widetilde{\alpha_i}}$}
    \put (310,40) {$\color{blue}{\widetilde{\beta_j}}$}
    
   \end{overpic}
   \caption{The developing image of $\beta_R$ in the grafting annuli.}
   \label{fig:5}
 \end{figure}

 \begin{figure}[p]
   \centering
   \begin{overpic}[width=16cm,clip]{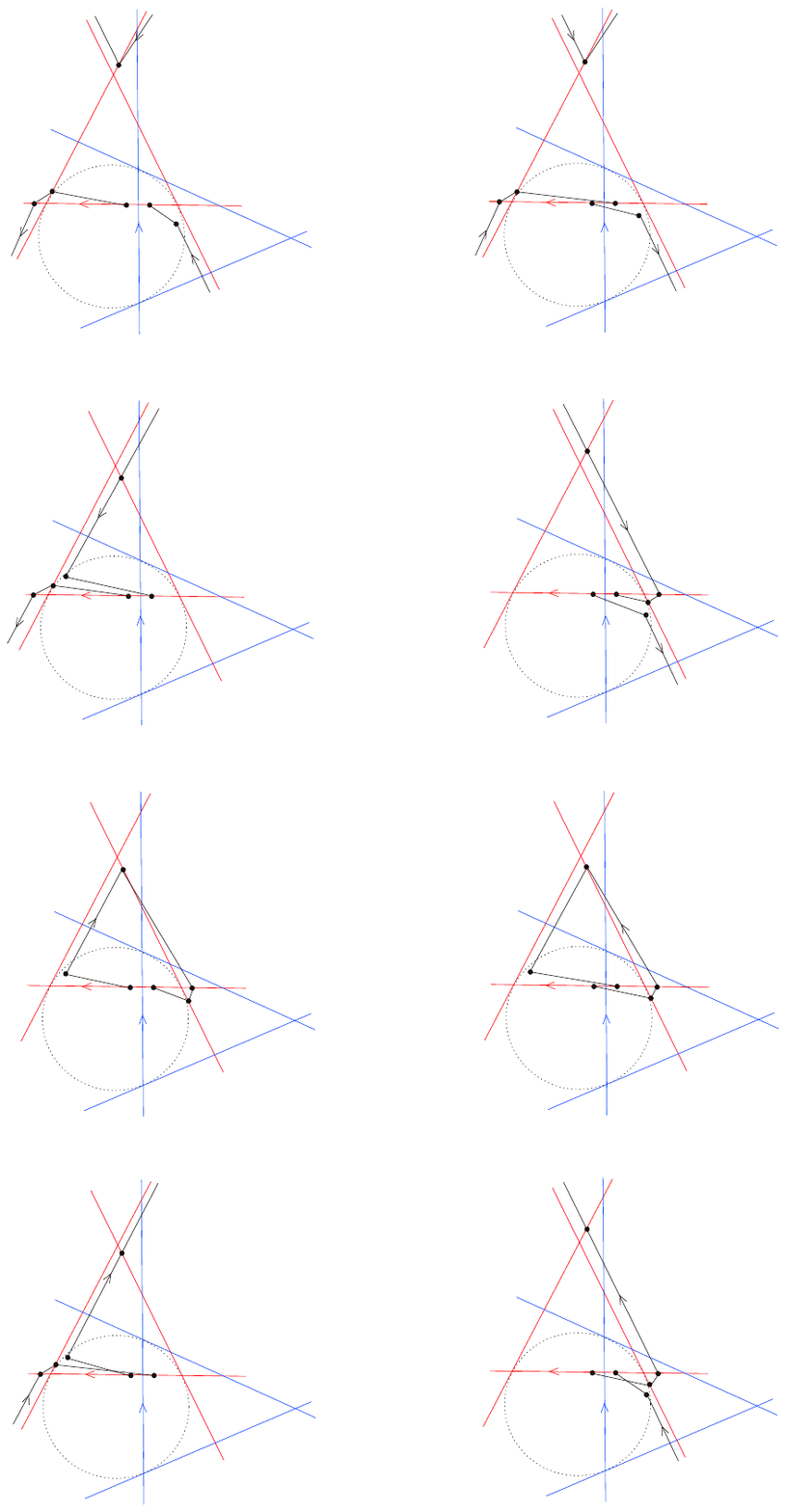}
   
   \put (120,490) {$H^{A_i}_{+xx}$}
    \put (100,547) {$\color{red}{\widetilde{\alpha_i}}$}
    \put (120,520) {$\color{blue}\widetilde{\beta_j}$}
    
   \put (310,490) {$H^{A_i}_{-xx}$}
    \put (290,547) {$\color{red}{\widetilde{\alpha_i}}$}
    \put (310,520) {$\color{blue}{\widetilde{\beta_j}}$}
    
   \put (120,325) {$H^{A_i}_{+xy}$}
    \put (100,387) {$\color{red}{\widetilde{\alpha_i}}$}
    \put (120,360) {$\color{blue}{\widetilde{\beta_j}}$}
    
   \put (310,325) {$H^{A_i}_{-xy}$}
    \put (290,387) {$\color{red}{\widetilde{\alpha_i}}$}
    \put (310,360) {$\color{blue}{\widetilde{\beta_j}}$}
    
   \put (120,167) {$H^{A_i}_{+yy}$}
    \put (100,225) {$\color{red}{\widetilde{\alpha_i}}$}
    \put (120,200) {$\color{blue}{\widetilde{\beta_j}}$}
    
   \put (310,167) {$H^{A_i}_{-yy}$}
    \put (290,225) {$\color{red}{\widetilde{\alpha_i}}$}
    \put (310,200) {$\color{blue}{\widetilde{\beta_j}}$}
    
   \put (120,10) {$H^{A_i}_{+yx}$}
    \put (100,65) {$\color{red}{\widetilde{\alpha_i}}$}
    \put (120,40) {$\color{blue}{\widetilde{\beta_j}}$}
    
   \put (310,10) {$H^{A_i}_{-yx}$}
    \put (290,65) {$\color{red}{\widetilde{\alpha_i}}$}
    \put (310,40) {$\color{blue}{\widetilde{\beta_j}}$}
    
   \end{overpic}
   \caption{The developing image of $\beta_L$ in the grafting annuli.}
   \label{fig:6}
 \end{figure}
 
 \newpage
 
Let us explain the construction when the multi-curve $\beta$ intersects $\alpha_i$ in more than one point. Let $\{p_l\}$ be the sets of intersection points of $\alpha_i$ and $\beta$. Then form the points $(p_l)'_R$ and $(p_l)''_R$ as before. Each simple arc connecting in grafting annulus is constructed in the same way as in the case where $\beta$ and $\alpha_i$ intersects at one point. Then, similarly to the above, assign the weight on each intersection point of $\beta$ and $\alpha$ and multiply the weights along each component $\beta_j$ of $\beta$. Let the obtained weight be denoted by ${\tau_j}^n$ for a primitive completely even word $\tau_j$. One can prove the following conditions;
 \begin{itemize}
  \item there is a lift  to $\widehat{T^{\rho(\beta_j)}_{\tau_j}}$ of the developing map restricted to $\widetilde{(\beta_j)_{R}}$, 
  \item this lift is injective.
\end{itemize}

 A key observation is that $\beta_R$ is a union of disjoint simple closed curves. In particular, we will check that the constructed arcs in $F^{\alpha_i}_{\pm *}$ are disjoint each other. In Figure \ref{fig:7}, we sketched $\beta_R$ in $F^{\alpha_i}_{+xx}$ in the case of two intersection points. 
 
 \begin{figure}[H]
  \centering
   \begin{overpic}[width=8cm,clip]{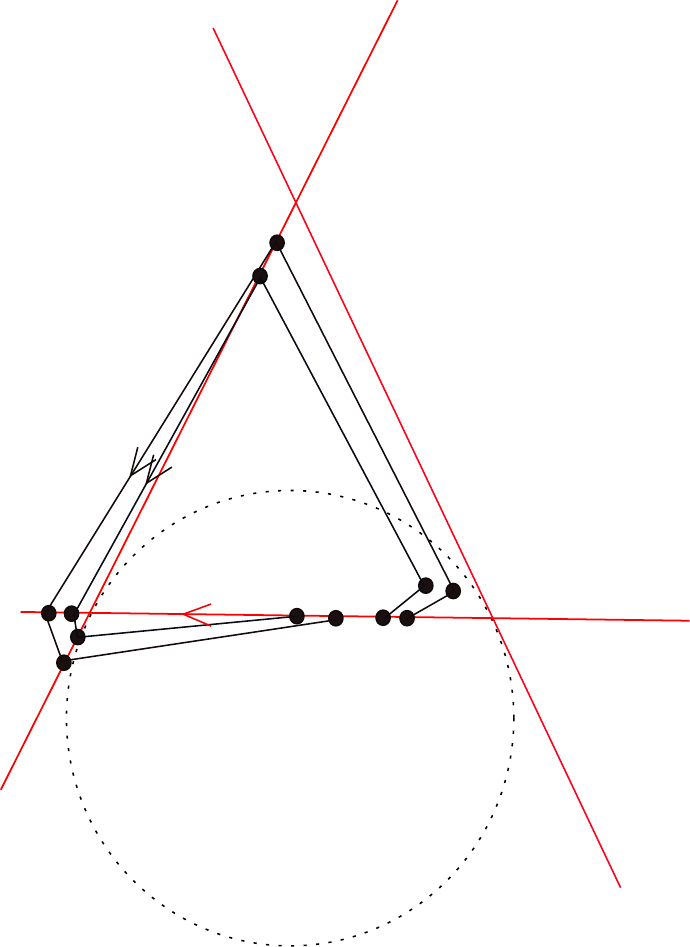}
   \put (40,115) {$\color{red}{\widetilde{\alpha_i}}$}
   
   \end{overpic} 
   \caption{The lifts of $\beta_R$ in $F^{\alpha_i}_{+xx}$ in the case where $\beta$ and $\alpha_i$ intersects in two points.}
   \label{fig:7}
 \end{figure}
 
 The lifts of simple arcs in $H^{A_i}_{\pm *}$ are disjoint, and hence such arcs are also disjoint when projecting to $F^{\alpha_i}_{\pm *}$. Thus, $\beta_R$ is a multi-curve.
 
 Finally, if $\beta$ has some components parallel to $\beta_j$, then we do the construction above for each copy of $\beta_j$. Then, we may move them by a small homotopy so that they no longer overlap, and arrange them side by side.

 The construction of $\beta_L$ is similar to that of $\beta_R$. We consider the points $p_L' \in \alpha_i'$ and $p_L'' \in \alpha_i''$ lying at distance $\epsilon$ from two copies $p' \in \alpha'$ and $p'' \in \alpha''$ to the left side of $p$ respectively. Then, we obtain $\beta_L$ in the grafting annulus which joins $p_L'$ and $p_L''$ by connecting the arcs in Figure \ref{fig:3}. The developing image of each simple arc in $H^{A_i}_{\pm *}$ is shown in Figure \ref{fig:6}. We may check that $\beta_L$ is a multi-curve such that there is an injective lift to $\widehat{T^{\rho(\beta_j)}_{\overline{\tau_j}}}$ of the developing map restricted to $\widetilde{(\beta_j)_L}$ for each $j \in J$.
 
 \begin{remark} \label{rem:3.1}
   Every component $(\beta_j)_R$, $(\beta_j)_L$ can be projectively embedded in the special $\mathbb{RP}^2$-torus $T^{\rho(\beta_j)}_{\tau_j}$, $T^{\rho(\beta_j)}_{\overline{\tau_j}}$ respectively. Therefore, although $\beta_R$ and $\beta_L$ are isotopic, they are not isotopic when regarded as graftable multi-curves.
 \end{remark}

 In the end, we proved the following.

 \begin{proposition} \label{prop:3.2}
   Let $\eta = (\{\alpha_i\}_{i \in I},\{w_i\}_{i \in I})$ be a grafting data. Given two homotopically transverse muti-curves $\alpha = \{\alpha_i\}_{i \in I}$ and $\beta = \{\beta_j\}_{j \in J}$, assign the weight on each intersection point of $\beta$ and $\alpha$ as explained above and multiply the weights along each components $\beta_j$ of $\beta$. Let the obtained weight be denoted by ${\tau_j}^n$ for a primitive complete even word $\tau_j$. Then, there is a graftable multi-curve $\beta_{R}$ and $\beta_{L}$ istopic to $\beta$ on $\mathbf{Gr}_{\eta}$ such that each component $(\beta_j)_R$, $(\beta_j)_L$ can be projectively embedded in the special $\mathbb{RP}^2$-torus $T^{\rho(\beta_j)}_{\tau_j}$, $T^{\rho(\beta_j)}_{\overline{\tau_j}}$ respectively.
 \end{proposition}

 \section{Grafting along the curves}

 We analyze the grafting along the graftable curves constructed in the previous section.

\subsection{The operation on grafting data}
  Given a grafting data $\eta = (\{\alpha_i\}_{i \in I},\{w_i\}_{i \in I})$ and a multi-curve $\beta = \{\beta_j\}_{j \in J}$ homotoipcally transverse to $\alpha = \{\alpha_i\}_{i \in I}$, we will obtain new grafting data $\eta \#_{R} \beta$ and $\eta \#_{L} \beta$.
  
 At each point of intersection $p \in \alpha_i \cap \beta_j$ choose points $p_N, p_S \in \alpha_i$ and $p_W, p_E \in \beta_j$ sufficiently close to $p$, so that these four points appear in the counterclockwise order $p_N$, $p_W$, $p_S$ and $p_E$. Furthermore, arrange the construction so that the direction (orientation) of $\beta_j$ be the same as the direction from $p_W$ to $p_E$. We construct a multi-curve by connecting these points with some arcs. This multi-curve consists of three types of arcs: (i) subarcs of $\alpha$, (ii) subarcs of $\beta$, and (iii) other arcs.

 (i) Let $\arc{p_N p_S}$ be a subarc of $\alpha_i$ connecting $p_N$ and $p_S$ through the point $p$. We obtain the desired subarcs by considering the complement $\alpha_i \setminus \bigcup \arc{p_N p_S}$ which is a union of subarcs connecting adjacent points on $\alpha_i$.

 (ii)  The components of $\alpha$ that intersect $\beta_j$ are ordered along the orientation of $\beta_j$ as $\alpha_{i_0},\alpha_{i_1}, \cdots ,\alpha_{i_{n-1}}$. We denote their respective intersection points by $p_{i_0},p_{i_1}, \cdots ,p_{i_{n-1}}$. As in the previous section, we assign a weight to each intersection point $p_{i_k}$ of $\beta_j$ and $\alpha$, and then multiply these weights along $\beta_j$. The resulting weight is a power of $\tau_j$, where $\tau_j$ is a completely even word. Suppose $\tau_j$ decomposes into a product of $m$ irreducible completely even words. We will connect $(p_{i_k})_E$ to $(p_{i_{k+m}})_W$ by a subarc of $\beta_j$ containing neither $p_{i_k}$ nor $p_{i_{k+m}}$ for each $k$.

 (iii) For each intersection point $p$ of $\alpha$ and $\beta$, we proceed as follows. To construct $\eta \#_{R} \beta$, we connect $p_N$ to $p_W$ and $p_S$ to $p_E$ as shown on the right in Figure \ref{fig:8}.
 To construct $\eta \#_{L} \beta$, we instead connect $p_N$ to $p_E$ and $p_S$ to $p_W$ as shown on the left in Figure \ref{fig:8}.

 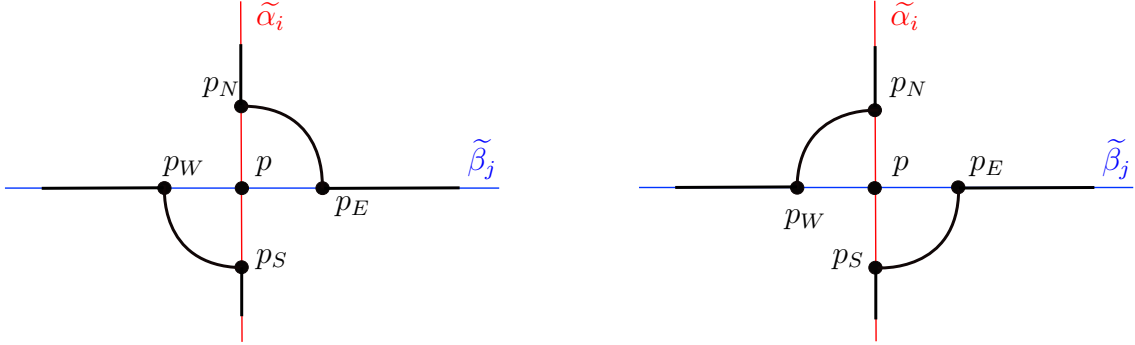
\begin{figure}[H]
  \centering
   \begin{overpic}[width=15cm,clip]{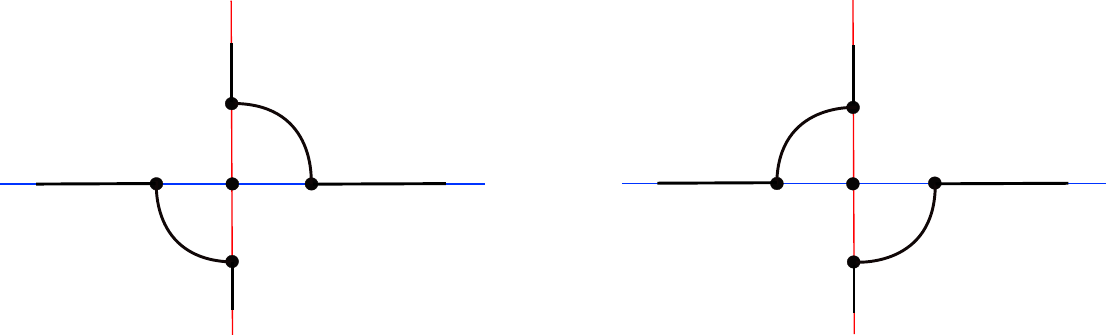}

   \put (95,120) {$\color{red}{\widetilde{\alpha_i}}$}
   \put (175,65) {$\color{blue}{\widetilde{\beta_j}}$}
   \put (95,65) {$p$}
   \put (75,95) {$p_N$}
   \put (95,30) {$p_S$}
   \put (125,50) {$p_E$}
   \put (60,65) {$p_W$}

   \put (335,120) {$\color{red}{\widetilde{\alpha_i}}$}
   \put (415,65) {$\color{blue}{\widetilde{\beta_j}}$}
   \put (336,65) {$p$}
   \put (335,95) {$p_N$}
   \put (313,30) {$p_S$}
   \put (365,65) {$p_E$}
   \put (295,45) {$p_W$}
   
   \end{overpic}
   \caption{the subsegments of $\eta \#_{R} \beta$ (right), of $\eta \#_{L} \beta$ (left).}
   \label{fig:8}
 \end{figure}

 The constructed curves consist non-oriented simple closed curves. We assign an orientation and a weight on each such simple closed curve $\gamma$ as follows. Take one component $\alpha_i$ of $\alpha$ that intersects $\gamma$, assign it an orientation compatible with $\alpha_i$, and give it a weight $w_i$. One may check that this assignment is well defined. By slightly perturbing the constructed curves via a homotopy, we obtained a new grafting data.

 \begin{remark} \label{rem:4.1}
   If we disregard the weights on grafting data, the operation $\#$ may be viewed as just a topological operation on multi-curves introduced in \cite{Luo:multi},\cite{Ito:exo}. Indeed, consider a multi-curve consisting $m$ parallel simple closed curve $\beta_j$ for each $j \in J$. The topological operation on this multi-curve and $\alpha$ yields the same multi-curve of $\eta \# \beta$.
 \end{remark}

 \subsection{Grafting along $\beta_R$ and $\beta_L$}

 Let $\eta = (\{\alpha_i\}_{i \in I},\{w_i\}_{i \in I})$ be a grafting data and $M$ be a $\mathbb{RP}^2$-surface corresponding to $\eta$ with Hitchin holonomy $\rho$. Recall that, given any multi-curve $\beta$ homotopically transverse with $\alpha = \{\alpha_i\}_{i \in I}$, we constructed graftable multi-curves $\beta_{R}$ and $\beta_{L}$ isotopic to $\beta$ in the previous section. Furthermore, each component $(\beta_j)_R$, $(\beta_j)_{L}$ is projectively embedded in the special $\mathbb{RP}^2$-torus $T^{\rho(\beta_j)}_{\tau_j}$, $T^{\rho(\beta_j)}_{\overline{\tau_j}}$ respectively. In this section, we shall use this special $\mathbb{RP}^2$-torus whenever grafting along $\beta_R$ and $\beta_L$. Accordingly, we will write $\mathbf{Gr}(M,((\beta_j)_R,T^{\rho(\beta_j)}_{\tau_j})_{j \in J})$ simply as $\mathbf{Gr}(M,\beta_R)$. Similarly, we abbreviate $\mathbf{Gr}(M,((\beta_j)_L,T^{\rho(\beta_j)}_{\overline{\tau_j}})_{j \in J})$ as $\mathbf{Gr}(M,\beta_L)$.

 We describe $\mathbf{Gr}(M,\beta_R)$ and $\mathbf{Gr}(M,\beta_L)$ as a grafting of a convex real projective surface. Recall that we denote the projective structure corresponding to a grafting data $\eta$ by $\mathbf{Gr}_{\eta}$.

 \begin{proposition} \label{prop:4.2}
   Given grafting data $\eta = (\{\alpha_i\}_{i \in I},\{w_i\}_{i \in I})$ and a multi-curve $\beta$ homotopically transverse to $\alpha = \{\alpha_i\}_{i \in I}$, let $\beta_R$ and $\beta_L$ denote graftable multi-curves on $\mathbf{Gr}_{\eta}$ constructed in the previous section. Then,
   
   \[\mathbf{Gr}(\mathbf{Gr}_{\eta},\beta_R) = \mathbf{Gr}_{\eta \#_{R} \beta},\] 
   \[\mathbf{Gr}(\mathbf{Gr}_{\eta},\beta_L) = \mathbf{Gr}_{\eta \#_{L} \beta}.\]
   
 \end{proposition}
 
 \begin{proof}
   
   We describe the grafting annuli of the projective structure $\sigma' = \mathbf{Gr}(\mathbf{Gr}_{\eta},\beta_{R})$. The description of the grafting annuli of $\mathbf{Gr}(\mathbf{Gr}_{\eta},\beta_{L})$ is similar.
 
   First of all, we may assume $\beta$ is a simple closed curve by applying the induction on $j \in J$. We denote by $\Omega$ the developing image of the universal cover of $\Sigma_c$.
   
   First, we choose one lift $\widetilde{\beta}$ of $\beta$, then we label the lifts of the components of $\alpha$ that meet $\widetilde{\beta}$ in order of intersection with $\widetilde{\beta}$ as $\{\widetilde{\alpha_{i}} \mid i \in \mathbb{Z}\}$. If $\Sigma$ be the real projective surface corresponding to the projective structure $\sigma = \mathbf{Gr}_{\eta}$, $\widetilde{\Sigma}$ is constructed by gluing the universal cover of the grafting annuli $\widehat{T^{\rho(\alpha_i)}_{w_i}} \setminus \widetilde{\alpha_i}$ to $\widetilde{\Sigma_c} \setminus \bigcup \widetilde{\alpha_i}$ and the developing map $D \colon \widetilde{\Sigma} \to \widetilde{\mathbb{RP}^2}$ is constructed as well. These grafting regions are called $\textit{bubbles}$, see Figure \ref{fig:9}.

   \begin{figure}[h]
   \centering
   \begin{overpic}[width=10cm,clip]{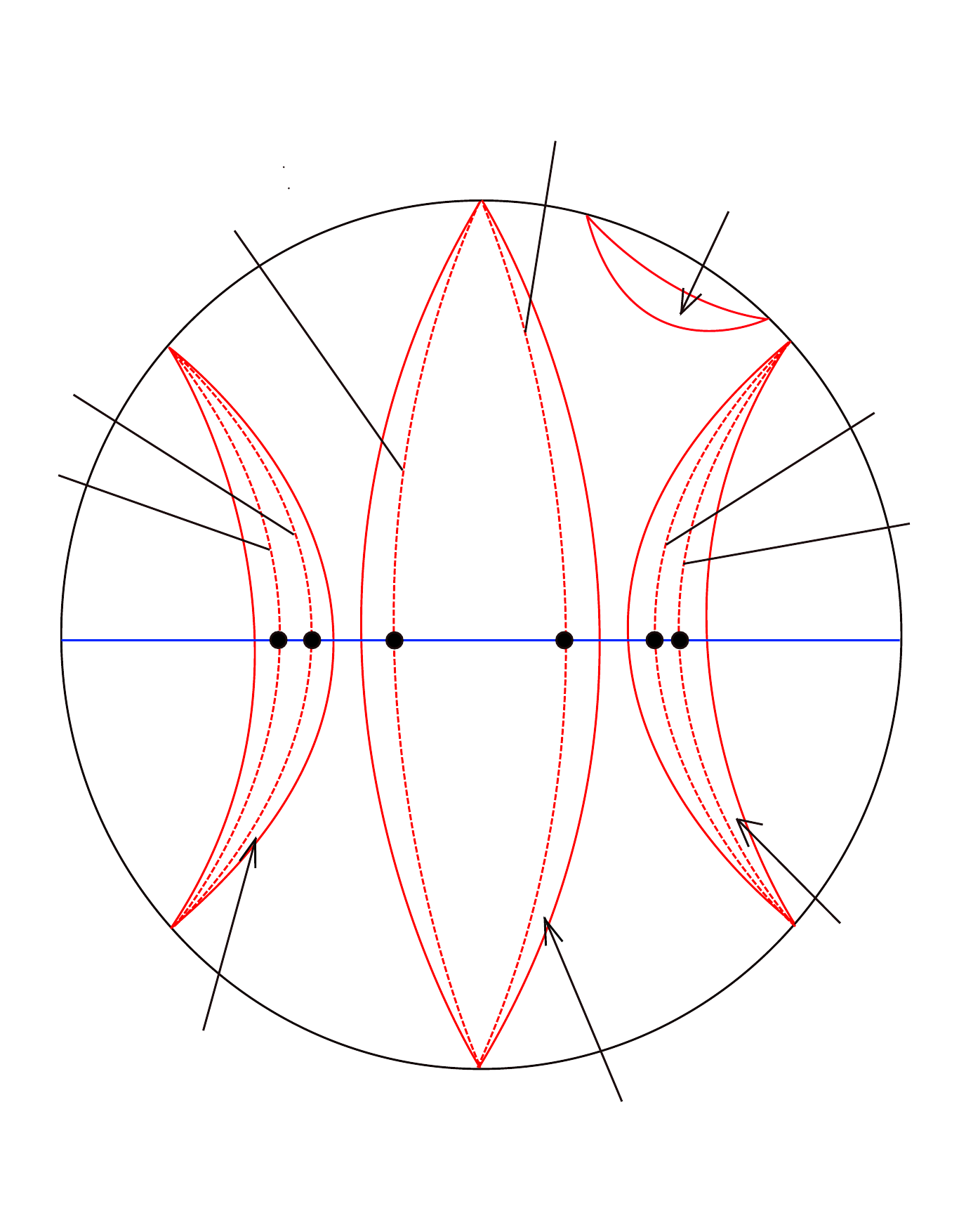}
   
   \put (20,50) {Bubble of ${\widetilde{\alpha}_{-1}}$}
    \put (-2,220) {$\widetilde{\alpha}_{-1}'$}
    \put (3,250) {$\widetilde{\alpha}_{-1}''$}
    \put (65,180) {$\widetilde{q}_{-1}$}
    \put (90,180) {$\widetilde{r}_{-1}$}
   
   \put (160,25) {Bubble of ${\widetilde{\alpha}}_0$}
    \put (60,300) {$\widetilde{\alpha}_0'$}
    \put (160,325) {$\widetilde{\alpha}_0''$}
    \put (110,180) {$\widetilde{q}_0$}
    \put (160,180) {$\widetilde{r}_0$}
    
   \put (240,80) {Bubble of ${\widetilde{\alpha}}_1$}
    \put (260,245) {$\widetilde{\alpha}_1'$}
    \put (270,210) {$\widetilde{\alpha}_1''$}
    \put (185,180) {$\widetilde{q}_1$}
    \put (200,180) {$\widetilde{r}_1$}
    
  \put (210,305) {Another bubble}
  \put (270,175) {${\widetilde{\beta}}_R$ or ${\widetilde{\beta}}_L$}
  
   \end{overpic}
   \caption{The universal cover of $\Sigma$}
   \label{fig:9}
   \end{figure}
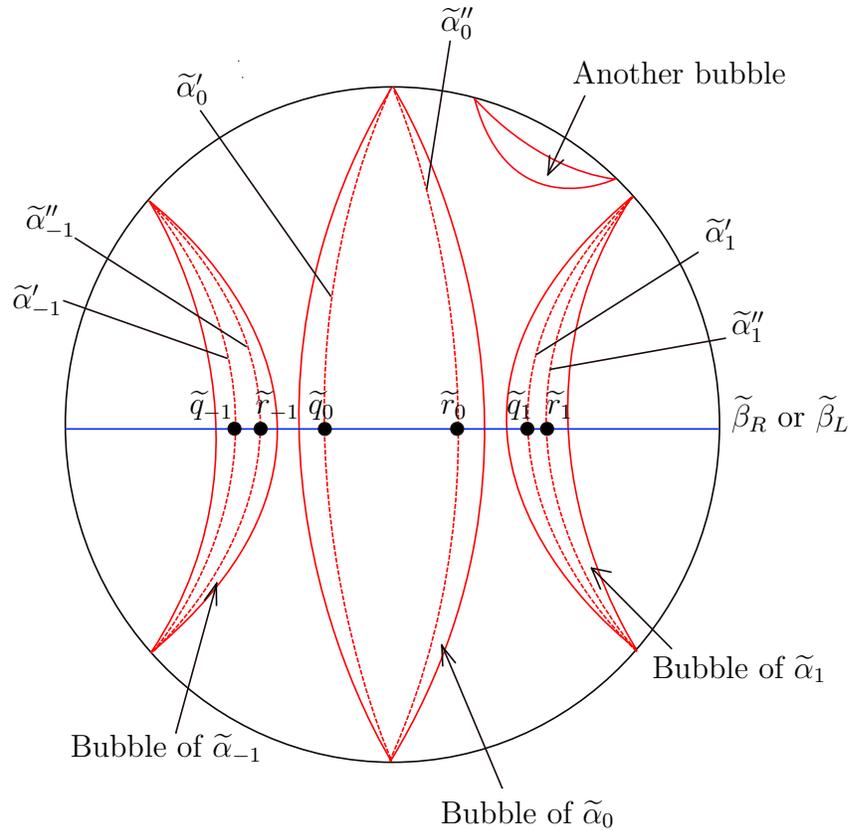

   In each bubble, there are two components of $D^{-1}(\partial \Omega)$, say $\widetilde{\alpha_i}'$ and $\widetilde{\alpha_i}''$ along the direction of $\beta$. The curve $\widetilde{\beta}_R$ intersects these curves $\widetilde{\alpha}_i'$ and $\widetilde{\alpha}_i''$ successively. For each $i$, we denote the intersection points of $\widetilde{\alpha}_i'$ and $\widetilde{\beta}_R$ by $\widetilde{q}_i$. Similarly, we denote the intersection points of $\widetilde{\alpha}_i''$ and $\widetilde{\beta}_R$ by $\widetilde{r}_i$.

   Let us analyze the projective structure $\sigma' = \mathbf{Gr}(\mathbf{Gr}_{\eta}, \beta_R)$ in more closely. The universal cover $\widetilde{\Sigma'}$ of the real projective surface corresponding to $\sigma'$ is obtained by cutting $\widetilde{\Sigma}$ along $\widetilde{\beta}_R$ and gluing a copy of $\widehat{T^{\rho(\beta)}_{\tau}} \setminus \widetilde{D}(\widetilde{\beta}_R)$. Let $D' \colon \widetilde{\Sigma}' \to \widetilde{\mathbb{RP}^2}$ be the developing map of $\sigma'$. We have two copies ${\widetilde{\beta}}^R_R$ and ${\widetilde{\beta}}^L_R$ of $\widetilde{\beta}_R$: ${\widetilde{\beta}}^R_R$ is the boundary component of the bubble on its right. Let ${\widetilde{q}}^R_i$ and ${\widetilde{q}}^L_i$ be the points corresponding to $\widetilde{q_i}$ lying in ${\widetilde{\beta}}^R_R$ and ${\widetilde{\beta}}^L_R$ respectively. Similarly, let ${\widetilde{r}}^R_i$ and ${\widetilde{r}}^L_i$ be the points corresponding to $\widetilde{r_i}$ lying in ${\widetilde{\beta}}^R_R$ and ${\widetilde{\beta}}^L_R$ respectively.

  If $\tau$ decomposes into $m$ irreducible completely even words, ${D'}^{-1}(\partial \Omega)$ in the bubble of ${\widetilde{\beta}}_R$ consists of the segments ${\widetilde{q}}^L_i$${\widetilde{q}}^R_{i+m}$ and ${\widetilde{r}}^L_i$${\widetilde{r}}^R_{i+m}$. We illustrated the segments in the bubble of $\beta_R$ which develops to $\partial \Omega$ in Figure \ref{fig:10}.

  \begin{figure}[h]
   \centering
   \begin{overpic}[width=10cm,clip]{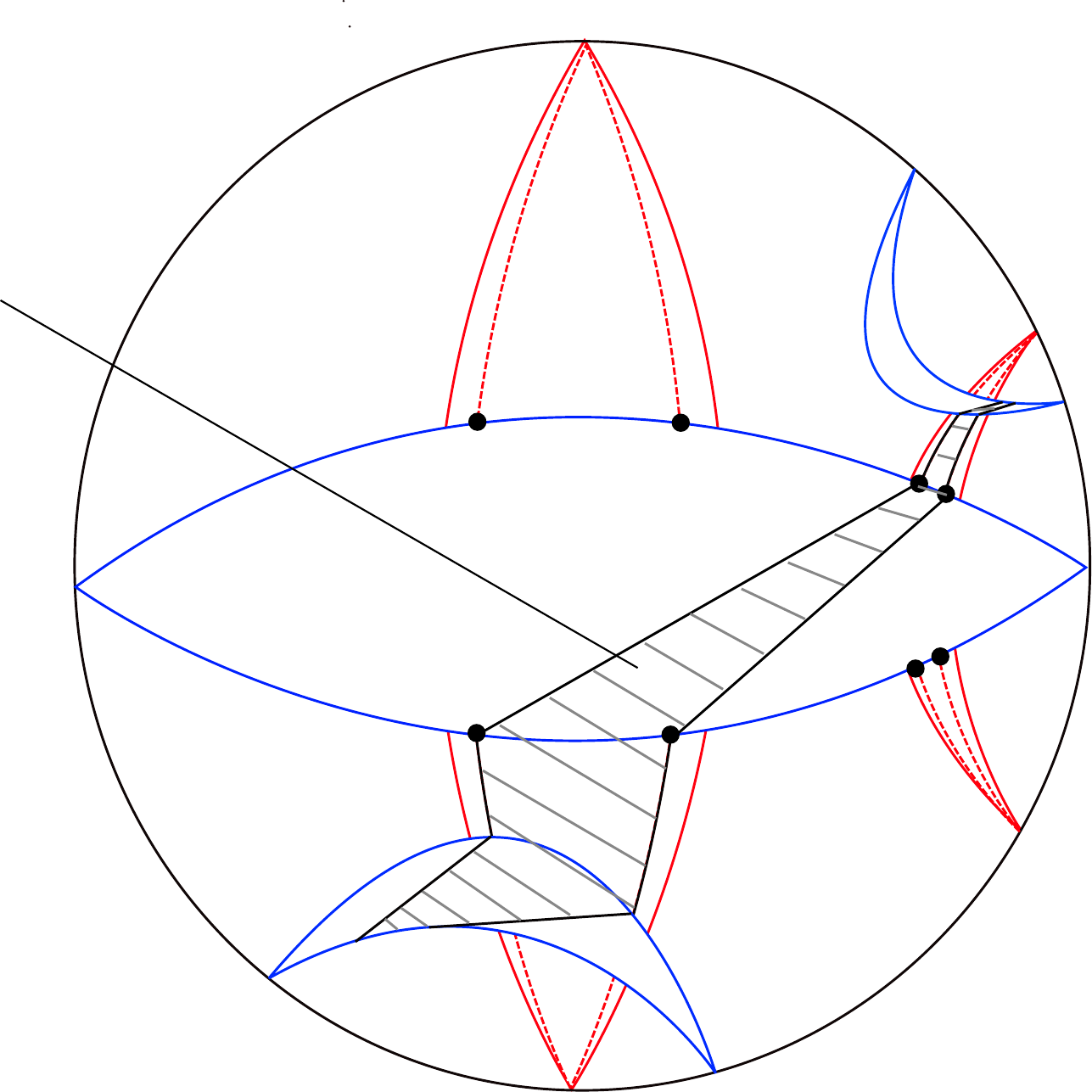}
   \put (115,180) {$\widetilde{q}^R_i$}
   \put (115,75) {$\widetilde{q}^L_i$}
   \put (177,180) {$\widetilde{r}^R_i$}
   \put (177,75) {$\widetilde{r}^L_i$}

   \put (215,170) {$\widetilde{q}^R_{i+m}$}
   \put (220,90) {$\widetilde{q}^L_{i+m}$}
   \put (255,155)  {$\widetilde{r}^R_{i+m}$}
   \put (255,105) {$\widetilde{r}^L_{i+m}$}

   \put (-15,210) {$\widetilde{E_i}$}
   \put (30,152) {${\widetilde{\beta}}^R_R$}
   \put (30,105) {${\widetilde{\beta}}^L_R$}
   
   \end{overpic}
   \caption{The universal cover of $\widetilde{\Sigma}'$ corresponding to the projective structure $\eta \#_R \beta$: the hatched region $\widetilde{E_i}$ develops to $\widetilde{\mathbb{RP}^2} \setminus \Omega$.}
   \label{fig:10}
  \end{figure}
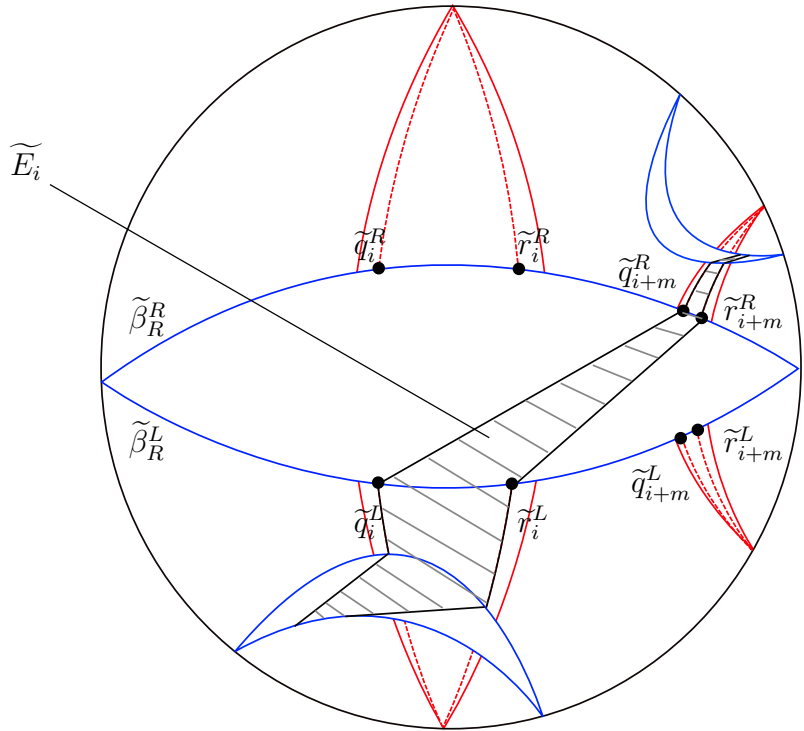
  
  Therefore, $D'^{-1}(\partial \Omega)$ is constructed by joining these segments and $\widetilde{\alpha_i}'$, $\widetilde{\alpha_i}''$. It is easy to see that $D'^{-1}(\widetilde{\mathbb{RP}^2} \setminus \Omega)$ is the region bounded by the curves $D'^{-1}(\partial \Omega)$.
  
  Finally, we will compute weights on grafting annuli. Let each component of $D'^{-1}(\widetilde{\mathbb{RP}^2} \setminus \Omega)$ that intersects with the bubble of $\beta_R$ and contains the points ${\widetilde{q}}^L_i$ and ${\widetilde{r}}^L_i$ be denoted by $\widetilde{E_i}$. We denote by $E_i$ the grafting annulus on the surface $\Sigma'$ covered by $\widetilde{E_i}$. We give $\partial E_i$ an orientation compatible with $\alpha_i$. Now we will compute the weight of the grafting annulus $E_i$. The curve ${\beta}^L_R \cap E_i$ intersects the simple closed curves which develop to the invariant lines of $\rho(\partial E_i)$ transversally. The axis (principal line) of $\rho(\partial E_i)$ intersects both $\widetilde{\alpha_i}$ and $\widetilde{\beta}$, as shown in Figure \ref{fig:11} below.
  
 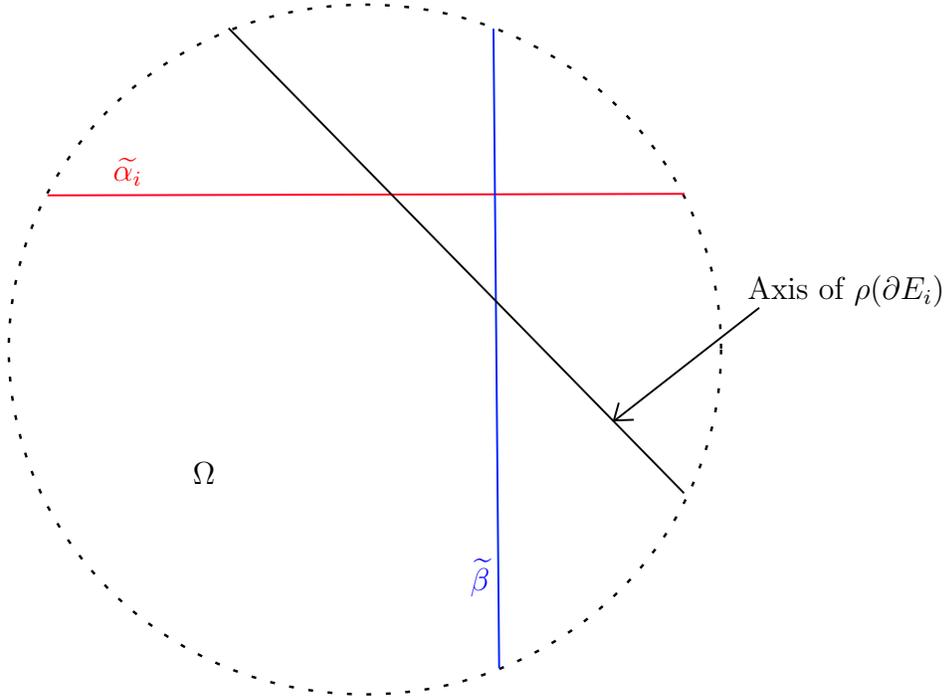
\begin{figure}[H]
   \centering
   \begin{overpic}[width=10cm,clip]{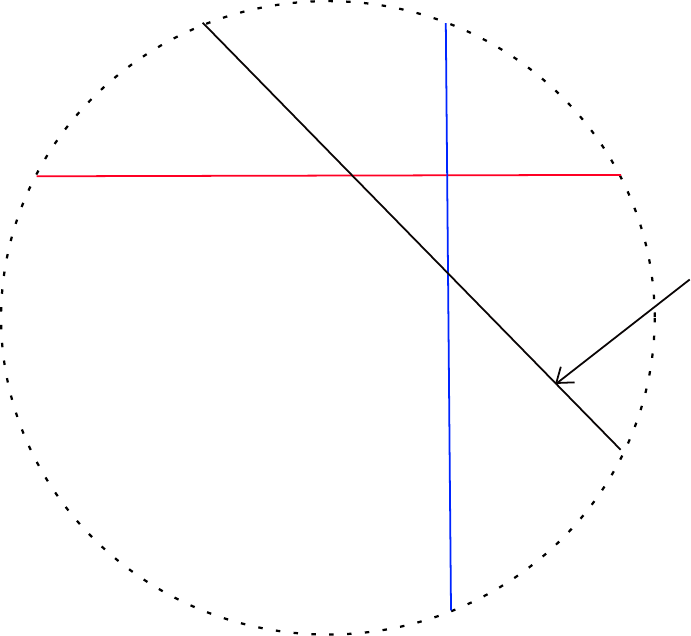}
   \put (40,195) {$\color{red}\widetilde{\alpha_i}$}
   \put (175,40) {$\color{blue}\widetilde{\beta}$}
   \put (280,150) {Axis of $\rho(\partial E_i)$}
   \put (70,80) {$\Omega$}
   \end{overpic}
   \caption{The axis of $\rho(\partial E_i)$}
   \label{fig:11}
 \end{figure}
  
  By observing the developing image of ${\widetilde{\beta}}^L_R \cap \widetilde{E_i}$ and recording the intersections with the invariant lines of $\rho(\partial E_i)$ in order, we can verify that the weight of $E_i$ is $w_i$.
 \end{proof}

 \begin{remark} \label{rem:4.3}
   We may obtain another projective structure $\sigma''$ by grafting $n_j$-fold covering of $T^{\rho(\beta_j)}_{\tau_j}$, $T^{\rho(\beta_j)}_{\overline{\tau_j}}$ along $(\beta_j)_R$, $(\beta_j)_L$ respectively for each $j \in J$. In this case, instead of $\beta$, we consider a multi-curve consisting of $n_j$ parallel copies of $\beta_j$ for each $j \in J$, then apply Proposition \ref{prop:4.2} to compute $\sigma''$.
 \end{remark}

 \subsection{The action of mapping class group on grafting data}

  Recall that a grafting data $\eta = (\{\alpha_i\}_{i \in I},\{w_i\}_{i \in I})$ is a pair of multi-curve $\alpha = \{\alpha_i\}_{i \in I}$ and weights $\{w_i\}_{i \in I}$ on $\alpha$.

  Let the mapping class group of an oriented surface $\Sigma$ be denoted by $\mathbf{Mod}(\Sigma)$. The mapping class group $\mathbf{Mod}(\Sigma)$ acts on the set of grafting data $\mathscr{GD}$ as follows. Given a grafting data $\eta = (\{\alpha_i\}_{i \in I},\{w_i\}_{i \in I})$ and $\phi \in \mathbf{Mod}(\Sigma)$, 
  \[\phi \cdot \eta \coloneq (\phi_{*}\{\alpha_i\}_{i \in I},\{w_i\}_{i \in I}).\]

  This action is related with the operation $\#$ on $\mathscr{GD}$:
  \[\phi \cdot (\eta \#_{R} \beta) = (\phi \cdot \eta) \#_{R} (\phi_{*}\beta),\]
  \[\phi \cdot (\eta \#_{L} \beta) = (\phi \cdot \eta) \#_{L} (\phi_{*}\beta).\]

  Hence, combining this fact with Proposition \ref{prop:4.2}, the following holds.

  \begin{corollary} \label{cor:4.4}
    Let $\eta = (\{\alpha_i\}_{i \in I},\{w_i\}_{i \in I})$ be a grafting data and $\beta$ be a multi-curve homotopically transverse with $\alpha = \{\alpha_i\}_{i \in I}$. Given $\phi \in \mathbf{Mod}(\Sigma)$,
  \[\mathbf{Gr}(\mathbf{Gr}_{\phi \cdot \eta},(\phi_{*}\beta)_{R}) = \mathbf{Gr}_{\phi \cdot (\eta \#_R \beta)},\]
  \[\mathbf{Gr}(\mathbf{Gr}_{\phi \cdot \eta},(\phi_{*}\beta)_{L}) = \mathbf{Gr}_{\phi \cdot (\eta \#_L \beta)}.\]
  \end{corollary}

  \section{The proof of the main theorem}
  
We prove Theorem A in this section.

  \subsection{Weight type of grafting data}

  Recall that grafting data $\eta = (\{\alpha_i\}_{i \in I},\{w_i\}_{i \in I})$ is a pair of the collection of disjoint simple closed curves $\{\alpha_i\}_{i \in I}$ and the collection of irreducible completely even words $\{w_i\}_{i \in I}$. In this section, if $\alpha_i$ and $\alpha_j$ are isotopic, one may treat them as a simple closed curve and assign it the product wights, namely $w_iw_j$ (or possibly $w_jw_i$ depending on the arrangement of $\alpha_i$ and $\alpha_j$).

  \begin{definition} \label{def:5.1}
    
    Let $\eta = (\{\alpha_i\}_{i \in I},\{w_i\}_{i \in I})$ be grafting data. The \textit{weight type} $WT(\eta)$ is a collection of the type of weights on the multi-curve $\{\alpha_i\}_{i \in I}$:
    \[WT(\eta) = \{w_i,{w_i}^{*} \mid i \in I \}\]
    
  \end{definition}

  \subsection{The proof of the main theorem}

  \begin{lemma} \label{lem:5.2}
    Let $w$ be a primitive completely even word satisfying $w = w^{*}$. Given grafting data $\eta_{\alpha}$ and $\eta_{\gamma}$ associated with multi-curves $\alpha$ and $\gamma$, where each component of $\alpha$ and $\gamma$ is assigned the weight $w$. Suppose $\alpha$ and $\gamma$ be two homotopically transversal multi-curves in $\Sigma$ such that each component of $\alpha$ intersects $\gamma$ vice versa. Then there exists a multi-curve $\beta$ intersecting $\alpha$ homotopically transversally and such that $\eta_{\gamma} = \eta_{\alpha} \#_R \beta$. Similarly, there exists a multi-curve $\beta'$ intersecting $\alpha$ homotopically transversally and such that $\eta_{\gamma} = \eta_{\alpha} \#_L \beta'$.
  \end{lemma}

  \begin{proof}
    In this case, we do not have to care for the orientation and the weights of grafting data. The operation $\#$ on grfating data descends to just an operation on multi-curves, as mentioned in \cite[\S 3.4]{CDF:gra} (see Remark \ref{rem:4.1}). This lemma is due to \cite[Lemma 4.1]{CDF:gra}.
  \end{proof}

  \begin{lemma} \label{lem:5.3}
    Let $w$ be a completely even word satisfying $w=w^{*}$. Given two grfating data $\eta_1$ and $\eta_2$ whose weights are all $w$, then $\mathbf{Gr}_{\eta_2}$ can be obtained from $\mathbf{Gr}_{\eta_1}$ by a composition of two multi-graftings.
  \end{lemma}

  \begin{proof}
    We can reduce to the case where $w$ is primitive. Indeed, if $w = {\tau}^n$, then $\tau$ also satisfies $\tau = {\tau}^{*}$.

   Let $\Sigma_c$ be the real convex projective surface with the holonomy $\rho$. We denote by $\alpha_1$ and $\alpha_2$ the two multi-geodesics with multiplicities on $\Sigma_c$ corresponding to the grafting annuli of $\mathbf{Gr}_{\eta_1}$ and $\mathbf{Gr}_{\eta_2}$. Consider the simple closed geodesic $\gamma$ on $\Sigma_c$ cutting all components of $\alpha_1$ and all components of $\alpha_2$. By Proposition \ref{prop:4.2} and Lemma \ref{lem:5.2}, there exist a multi-curve $\beta_1$ on $\mathbf{Gr}_{\eta_1}$ and a multi-curve $\beta_2$ on $\mathbf{Gr}_{(\gamma, w)}$ such that $\mathbf{Gr}(\mathbf{Gr}_{\eta_1},(\beta_1)_R) = \mathbf{Gr}_{(\gamma, w)}$ and $\mathbf{Gr}(\mathbf{Gr}_{(\gamma, w)},(\beta_2)_R) = \mathbf{Gr}_{\eta_2}$.
  \end{proof}

 The following proposition is a generalization of the above lemma.
 
  \begin{proposition} \label{prop:5.4}
    Let $w$ be a completely even word. Given two grfating data $\eta_1$ and $\eta_2$ whose weights are all $w$, then $\mathbf{Gr}_{\eta_2}$ can be obtained from $\mathbf{Gr}_{\eta_1}$ by a composition of multi-graftings.
  \end{proposition}

  \begin{proof}
    Let $\sigma_i = \mathbf{Gr}_{\eta_i}$ and $\alpha_i$ be multi-curves of $\eta_i$. If $w = (w')^n$ for some $n >1$ and primitive completely even word $w'$, then we can regard the weight of grafting data $\eta_1$ and $\eta_2$ as $w'$.  Therefore, we may assume $w$ is primitive. Furthermore, we may assume $\alpha_2$ is a simple closed curve since grafting is possible along each component of $\alpha_2$.
    
    By grafting along each component of $\alpha_1$ on $\sigma_1=\sigma^{(0)}_1$, we obtain a projective structure $\sigma^{(1)}_1$ in which all weights take the value $ww^{*}$.  Using Lemma \ref{lem:5.3}, this gives a projective structure $\sigma^{(2)}_1$ corresponding to grafting data $(\alpha,ww^{*})$ for any non-trivial simple closed curve $\alpha$ by a composition of multi-graftings.

    We take $\alpha$ to be the non-separating simple closed curve illustrated in Figure \ref{fig:12}. 
    
 \begin{figure}[h]
   \centering
   \begin{overpic}[width=10cm,clip]{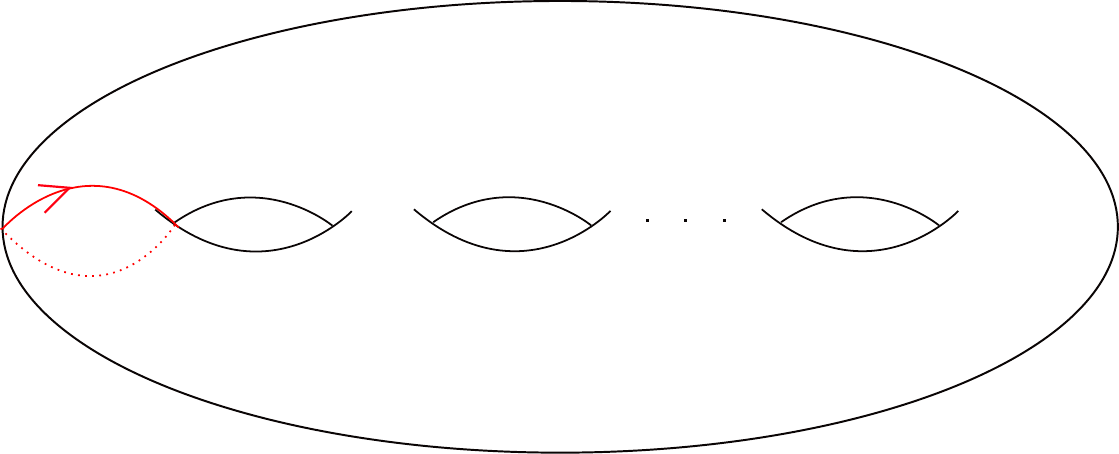}
    \put (20,70) {\color{red}$\alpha$}
   \end{overpic}
   \caption{The simple closed curve $\alpha$}
   \label{fig:12}
 \end{figure}

 We will construct a simple closed curve $\gamma$ from $\sigma^{(2)}_1$, via multi-graftings, such that there is an orientation-preserving diffeomorphism $\phi$ which sends $\gamma$ to $\alpha_2$. By Corollary \ref{cor:4.4}, $\mathbf{Gr}_{\eta_2}$ can be obtained from $\mathbf{Gr}_{(\phi_{*}\alpha, ww^{*})}$ via multi-graftings. This completes the proof.
 
  We divide the argument into two cases: (i) $\alpha_2$ is non-separating, and (ii) $\alpha_2$ is separating.
    
    (i) Suppose that $\alpha_2$ is non-separating. First, we will realize two disjoint simple closed curve $\gamma_1$ and $\gamma_2$ with a weight $w$. Let $\alpha'$ be a non-separating simple closed curve which does not intersect with $\alpha$. Let $\beta$ be the simple closed curve intersecting $\alpha$ and $\alpha'$ exactly once respectively illustrated in Figure \ref{fig:13}.
    
  \begin{figure}[H]
   \centering
   \begin{overpic}[width=10cm,clip]{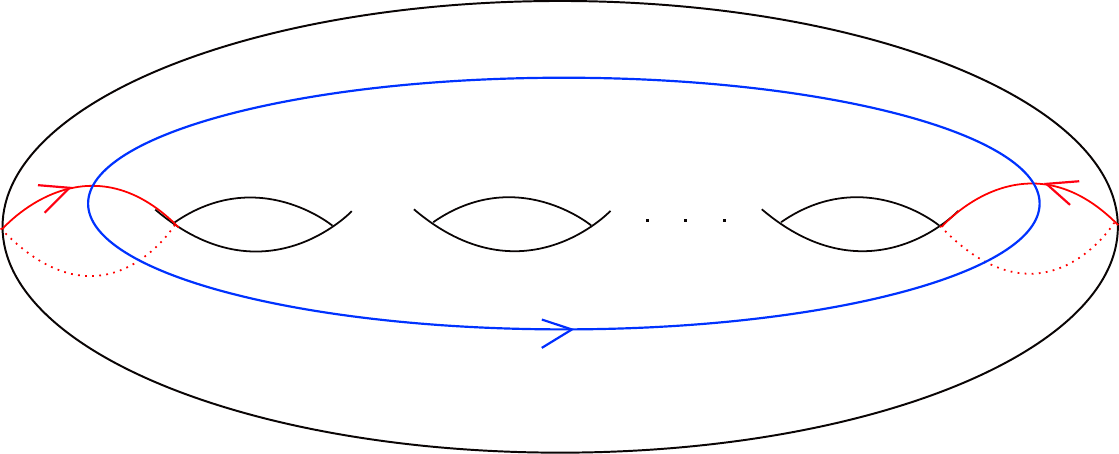}
    \put (10,70) {\color{red}$\alpha$}
    \put (270,55) {\color{red}$\alpha'$}
    \put (150,20) {\color{blue}$\beta$}
   \end{overpic}
   \caption{The simple closed curves $\alpha'$ and $\beta$}
   \label{fig:13}
 \end{figure}

  By grafting along ($\alpha$, $T^{\rho(\alpha)}_{w*}$) and ($\alpha'$, $T^{\rho(\alpha')}_{w}$), we obtain another projective structure $\sigma^{(3)}_1$ corresponding to grafting data $\{(\alpha,w^{*}ww^{*}),(\alpha',w)\}$. We then obtain a projective structure $\sigma^{(4)}_1 = \mathbf{Gr}(\sigma^{(3)}_1,(\beta_R,T^{\rho(\beta)}_{ww^{*}}))$ by grafting along $\beta_R$. By Proposition \ref{prop:4.2}, $\sigma^{(4)}_1$ corresponds to the grafting data $\{(\gamma_1,w),(\gamma_2,w)\}$ where $\gamma_1$ and $\gamma_2$ are illustrated in Figure \ref{fig:14}.
    
 \begin{figure}[h]
   \centering
   \begin{overpic}[width=10cm,clip]{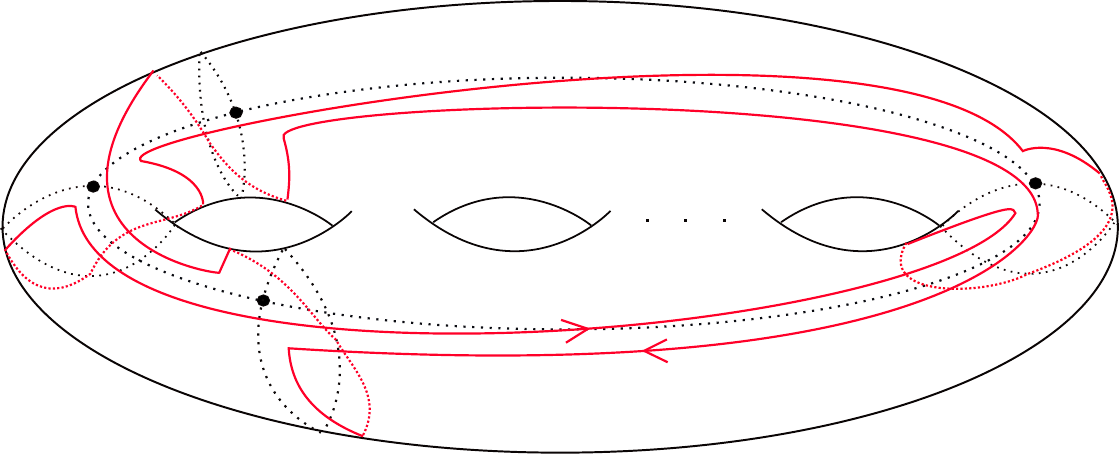}
    \put (130,40) {\color{red}$\gamma_1$}
    \put (190,20) {\color{red}$\gamma_2$}
   \end{overpic}
   \caption{The simple closed curves $\gamma_1$ and $\gamma_2$}
   \label{fig:14}
 \end{figure}

 Next, we will transform $\gamma_1$ and $\gamma_2$ to a simple closed curve $\gamma$.
    
 There is a non-separating simple closed curve $\gamma$ such that $\hat{i}(\gamma_1,\gamma)=i(\gamma_1,\gamma)=\hat{i}(\gamma_2,\gamma)=i(\gamma_2,\gamma)=1$ (Figure \ref{fig:15}).
 
 \begin{figure}[h]
   \centering
   \begin{overpic}[width=10cm,clip]{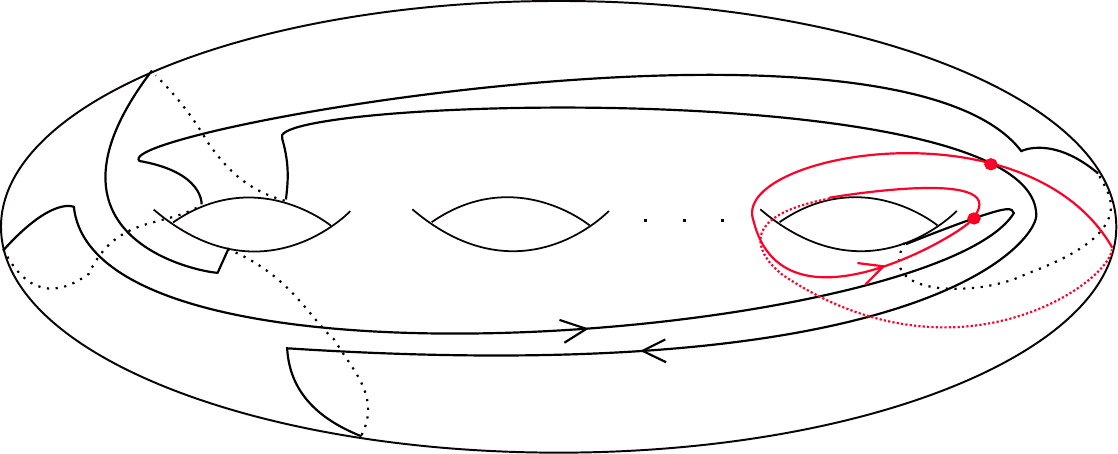}
    \put (185,45) {\color{red}{$\gamma$}}
    \put (130,40) {$\gamma_1$}
    \put (190,20) {$\gamma_2$}
   \end{overpic}
   \caption{The simple closed curve $\gamma$}
   \label{fig:15}
 \end{figure}
 
 Then, we construct a simple closed curve $\delta$ by applying Dehn twists along ${\gamma_1}^{-1}$ and ${\gamma_2}^{-1}$ to $\gamma$ simultaneously. The grafting along $\delta_{R}$ yields $\mathbf{Gr}_{(\gamma,w)}=\mathbf{Gr}(\sigma^{(4)}_1,\delta_R)$. (Compare with \cite[Lemma 4.1]{CDF:gra}.)

    (ii) Next suppose $\alpha_2$ is separating. Let $\alpha'$ be a separating simple closed curve which does not intersect with $\alpha$. Let $\beta$ be the simple closed curve that intersects $\alpha$ once and $\alpha'$ twice respectively illustrated in Figure \ref{fig:16}.
    
  \begin{figure}[h]
   \centering
   \begin{overpic}[width=10cm,clip]{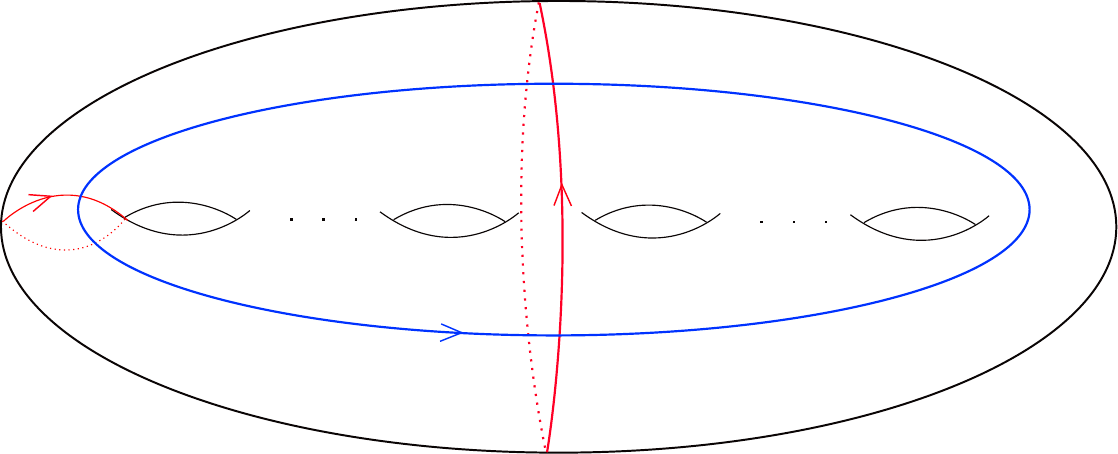}
   \put (10,70) {\color{red}$\alpha$}
   \put (145,70) {\color{red}$\alpha'$}
   \put (110,20) {\color{blue}$\beta$}
   \end{overpic}
   \caption{The simple closed curve $\alpha'$ and $\beta$}
   \label{fig:16}
 \end{figure}
    
    By grafting along $(\alpha', T^{\rho(\alpha')}_{w})$, we obtain another projective structure $\sigma^{3}_1$ corresponding to grafting data $\{(\alpha,ww^{*}),(\alpha',w)\}$. We then obtain a projective structure $\sigma^{(3)}_1 = \mathbf{Gr}(\sigma^{(2)}_1,\beta_R)$. By Proposition \ref{prop:4.2}, $\sigma^{(3)}_{1}$ corresponds to the grafting data $(\gamma,w)$ where $\gamma$ is the separating simple closed curve illustrated in Figure \ref{fig:17}.

 \begin{figure}[h]
   \centering
   \begin{overpic}[width=10cm,clip]{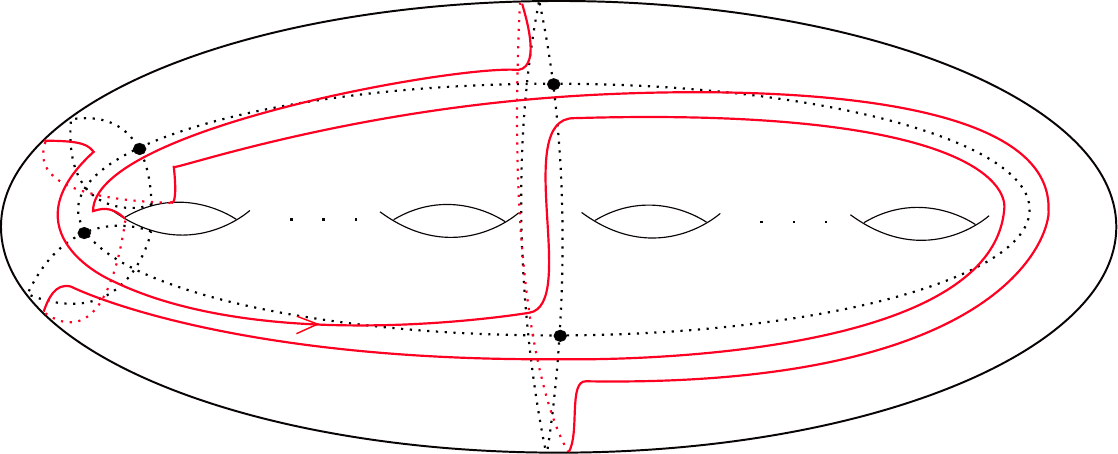}
    \put (80,40) {\color{red}$\gamma$}
   \end{overpic}
   \caption{The simple closed curve $\gamma$}
   \label{fig:17}
 \end{figure}

 We exchange $w$ and $w^{*}$ and carry out the same argument above to construct ${\gamma}^{-1}$. Every simple closed curve may be obtained by mapping $\gamma$ or $\gamma^{-1}$ via an orientation-preserving diffeomorphism.
  \end{proof}

  Now we shall prove the main theorem. Let $\Sigma$ be a closed oriented surface with genus $g$. 

 \begin{theorem} \label{theorem}
   Let $\eta_1$ and $\eta_2$ be grafting data on $\Sigma$ with the same weight types. Let $\mathbf{Gr}_{\eta_1}$ and $\mathbf{Gr}_{\eta_2}$ be projective structures sharing the same Hitchin holonoimy $\rho$. Then $\mathbf{Gr}_{\eta_2}$ can be obtained from $\mathbf{Gr}_{\eta_1}$ by a composition of at most $6g$ multi-graftings.
 \end{theorem}

 \begin{proof}
   It suffices to show the following: Let $\gamma = \{\gamma_1,\gamma_2, \cdots ,\gamma_{3g-3}\}$ be a pants-decomposition. Given grafting data $\eta_2 = (\{\gamma_i\}^{3g-3}_{i=1},\{w_i\}^{3g-3}_{i=1})$, there is a grafting data $\eta_1' = (\{\alpha_i\}_{i=1}^{3g-3},\{\tau_i\}_{i=1}^{3g-3})$ such that 

   \begin{itemize}
     \item $\alpha_1,\alpha_2, \cdots \alpha_{3g-3}$ are mutually disjoint and parallel, and sit on $\Sigma$ in the order of the indices,
     \item $\tau_i = w_i$ or $\tau_i = {w_i}^{*}$,
     \item $\mathbf{Gr}_{\eta_2}$ can be obtained from $\mathbf{Gr}_{\eta_1'}$ by a composition of multi-graftings.
   \end{itemize}

   This is because, by grafting with appropriate weights along each component of the multi-curve of $\eta_1$, we can make all the weights equal to $\tau_1 \cdots \tau_{3g-3}$. By Proposition \ref{prop:5.4}, the multi-curve can also be reduced to an arbitrary simple closed curve.

  We can suppose $\gamma_1$ and $\gamma_{3g-3}$ are non-separating simple cosed curves. Furthermore, we may assume that there is a simple closed curve $\gamma$ such that for some $1 \le k < 3g-3$,
  
  $i(\gamma_i,\gamma)=0$ ($i=1, \cdots ,k$)

  $i(\gamma_i,\gamma)=1$ ($i=k+1, \cdots ,3g-3$).

  By reordering the indices, we may further assume that for any $1 \le i < 3g-3$, when the surface $\Sigma$ is cut along $\gamma_1, \cdots ,\gamma_i$, exactly one of the resulting components is not a pair-of-pants. As an example, we illustrate the case of genus $3$ in Figure \ref{fig:pair-of-pants}.

  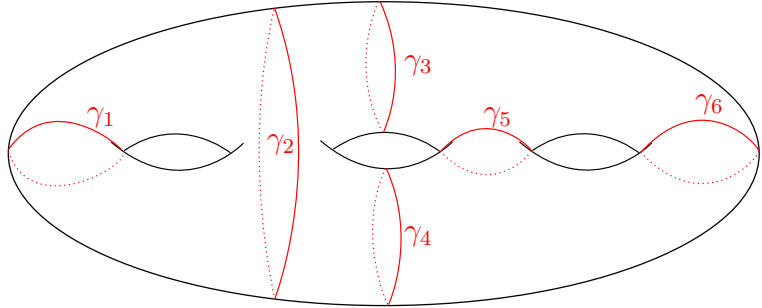
\begin{figure}[h]
   \centering
   \begin{overpic}[width=10cm,clip]{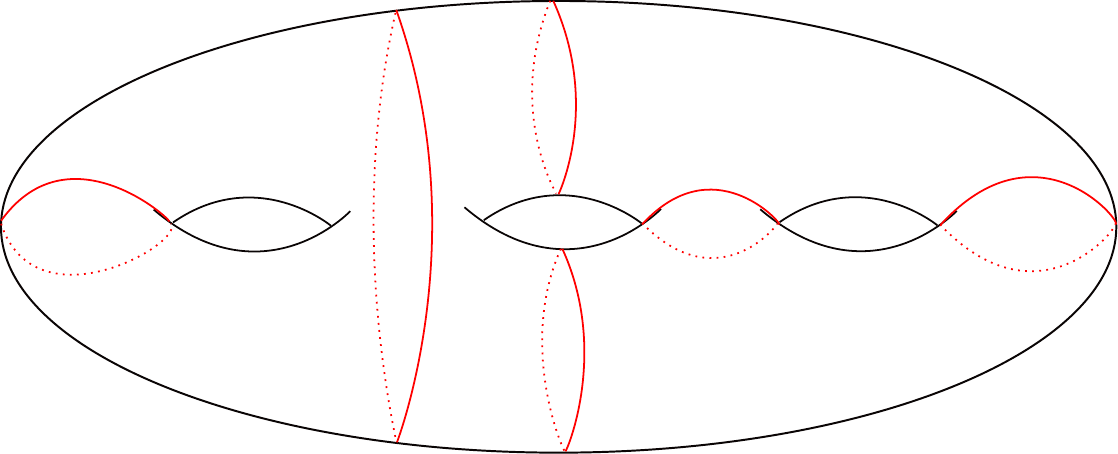}
    \put (30,70) {\color{red} $\gamma_1$}
    \put (98,60) {\color{red} $\gamma_2$}
    \put (150,90) {\color{red} $\gamma_3$}
    \put (150,25) {\color{red} $\gamma_4$}
    \put (180,70) {\color{red} $\gamma_5$}
    \put (260,75) {\color{red} $\gamma_6$}
   \end{overpic}
   \caption{A pair-of-pants $\{\gamma_i\}^{6}_{i=1}$ on the closed surface with genus 3.}
   \label{fig:pair-of-pants}
  \end{figure}
  
  Given grafting data $\eta_1'=(\{\alpha_i\},\{\tau_i\})$ such that $\alpha_i=\delta_0$ for the non-separating simple closed curve $\delta_0$ illustrated in Figure \ref{fig:19}.
  
  \begin{figure}[h]
   \centering
   \begin{overpic}[width=10cm,clip]{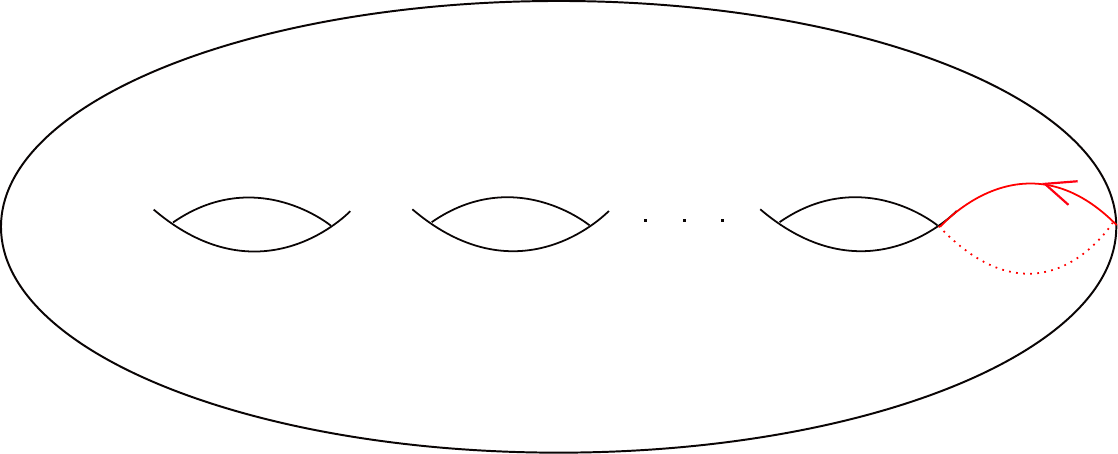}
   \put (260,75) {\color{red} $\delta_0$}
   \end{overpic}
   \caption{The simple closed curve $\delta_0$}
   \label{fig:19}
  \end{figure}
  
  For each $i = 0,1, \cdots$, define a completely even word $a_i$ by $a_i = \tau_{i+1} \cdots \tau_{3g-3}$. We will construct from grafting data $\mathbf{Gr}_{\eta_1'}$, via multi-graftings, a pants-decomposition $\{\overline{\gamma}_1, \cdots ,\overline{\gamma}_{3g-3}\}$ such that there is an orientation-preserving diffeomorphism $\phi$ which sends $\overline{\gamma_i}$ to the corresponding $\gamma_i$. By Corollary \ref{cor:4.4}, $\mathbf{Gr}_{\eta_2}$ can be obtained from $\mathbf{Gr}_{\phi \cdot \eta_1'}$ via multi-graftings. This completes the proof.

 \begin{enumerate}
   
   \item[(i)] First we construct simple closed curves $\overline{\gamma}_1, \cdots ,\overline{\gamma}_k$ inductively:
     \begin{enumerate}
     
       \item[(a)] Construct $\overline{\gamma}_1$ as in the proof of Proposition \ref{prop:5.4}: Let ${\delta_0}'$ be a non-separating simple closed curve which does not intersect with $\delta_0$. Let $\beta_0$ be the simple closed curves intersecting $\delta_0$ and $\delta_0'$ exactly once respectively illustrated in Figure \ref{fig:20}.

  \begin{figure}[h]
   \centering
   \begin{overpic}[width=10cm,clip]{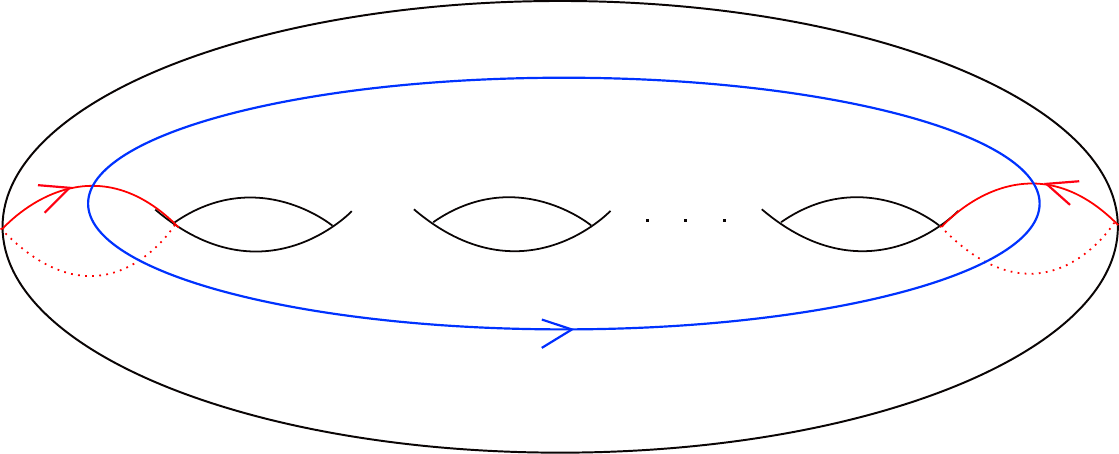}
   \put (260,75) {\color{red} $\delta_0$}
   \put (10,75) {\color{red} $\delta_0'$}
   \put (140,15) {\color{blue} $\beta_0$}
   \end{overpic}
   \caption{The simple closed curves $\delta_0'$ and $\beta_0$}
   \label{fig:20}
  \end{figure}
  
  We obtain a projective structure ${\sigma_0}'$ corresponding to grafting data $\{(\delta_0,a_1a_0),({\delta_0}',\tau_1)\}$ from $\sigma_0 = \mathbf{Gr}_{\eta_1'}$ by grafting along $\delta_0$ and ${\delta_0}'$. Then we may obtain another projective structure $\sigma_1 = \mathbf{Gr}({\sigma_0}',((\beta_0)_R,T^{\rho(\beta)}_{a_0}))$. By Proposition \ref{prop:4.2}, $\sigma_1$ corresponds to the grafting data $\{(\overline{\gamma}_1,\tau_1),(\delta_1,a_1)\}$ where $\overline{\gamma_1}$ and $\delta_1$ are illustrated as in Figure \ref{fig:21}.

  \begin{figure}[h]
   \centering
   \begin{overpic}[width=10cm,clip]{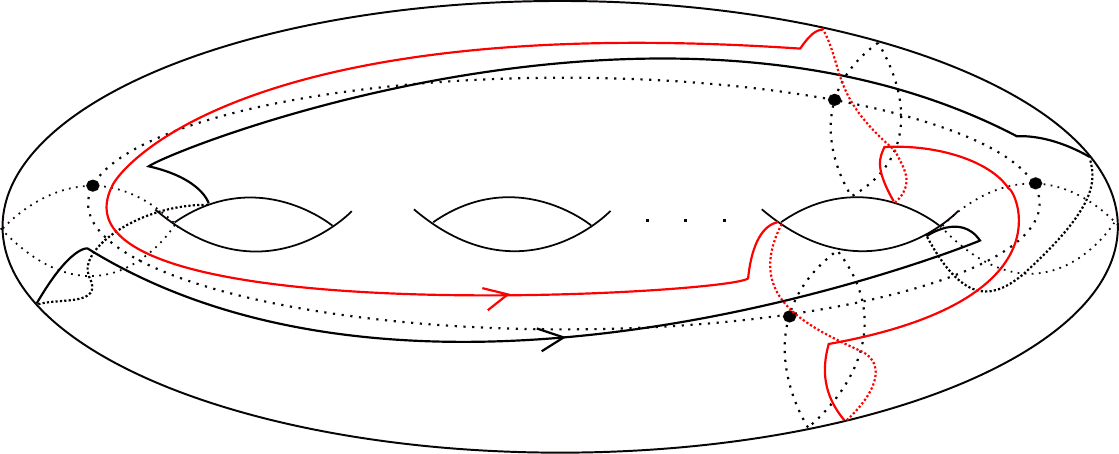}
   \put (140,15) { $\overline{\gamma}_1$}
   \put (95,45) {\color{red} $\delta_1$}
   \end{overpic}
   \caption{The simple closed curves $\overline{\gamma}_1$ and $\delta_1$}
   \label{fig:21}
  \end{figure}
      
      Cutting the surface along the constructed curve $\overline{\gamma}_1$ produces a subsurface $\Sigma_1$.

       \item[(b)] For $i > 1$, construct $\overline{\gamma}_i$ based on the previous constructed projective structure $\sigma_{i-1}$ and subsurfaces $\Sigma_{i-1}$: the projective structure $\sigma_{i-1}$ corresponds to the grafting data $\{(\overline{\gamma}_1,\tau_1), \cdots ,(\overline{\gamma}_{i-1},\tau_{i-1}),(\delta_{i-1},a_{i-1})\}$. $\delta_{i-1}$ is the non-separating simple closed curve on $\Sigma_{i-1}$. We consider two cases.

        \begin{enumerate}
           \item[(b-1)] there exists a pair-of-pants with boundary components $\gamma_{i_1},\gamma_{i_2},\gamma_{i}$ ($i_1,i_2 < i$): We place $\delta_{i-1}$ in the position on $\Sigma_{i-1}$ illustrated as in Figure \ref{fig:22}. Let $\delta_{i-1}'$ be the separate simple closed curve on $\Sigma_{i-1}$ such that one component of $\Sigma_{i-1}|\delta_{i-1}'$ is a pair-of-pants with boundary $\overline{\gamma}_{i_1},\overline{\gamma}_{i_2},\delta_{i-1}'$, see Figure \ref{fig:22}.
             
  \begin{figure}[h]
   \centering
   \begin{overpic}[width=10cm,clip]{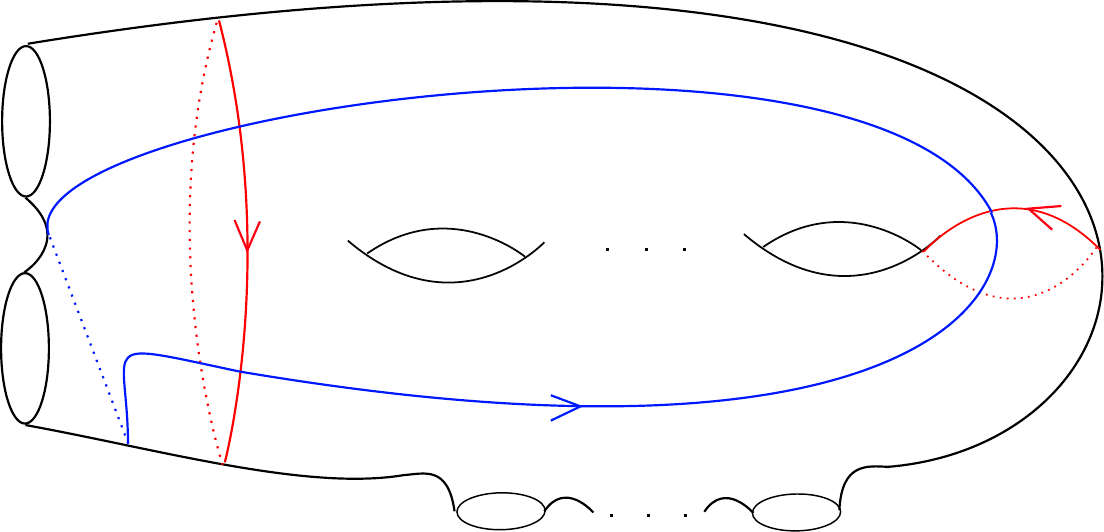}
   \put (-15,105) {$\overline{\gamma}_{i_1}$}
   \put (-15,45) {$\overline{\gamma}_{i_2}$}
   \put (260,90) {\color{red} $\delta_{i-1}$}
   \put (70,80) {\color{red} $\delta_{i-1}'$}
   \put (120,40) {\color{blue} $\beta_{i-1}$}
   \end{overpic}
   \caption{The simple closed curves $\delta_{i-1}$, $\delta_{i-1}'$ and $\beta_{i-1}$ on $\Sigma_{i-1}$}
   \label{fig:22}
  \end{figure}
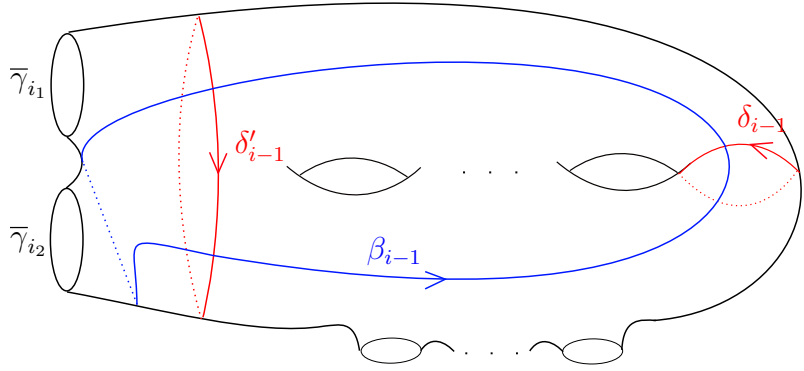     
             
 We obtain a projective structure $\sigma_{i-1}'$ corresponding to grafting data \[\{(\overline{\gamma_1},\tau_1), \cdots ,(\overline{\gamma}_{i-1},\tau_{i-1}),(\delta_{i-1},a_i \tau_i {\tau_i}^{*} a_i),(\delta_{i-1}',\tau_i)\}\] by grafting along $\delta_{i-1}$ and $\delta_{i-1}'$. Let $\beta_{i-1}$ be the simple closed curve intersecting $\delta_{i-1}$ once and $\delta_{i-1}'$ twice respectively illustrated in Figure \ref{fig:22}. Then, we may obtain another projective structure $\sigma_{i} = \mathbf{Gr}(\sigma_{i-1}',((\beta_{i-1})_R,T^{\rho(\beta_{i-1})}_{\tau_i {\tau_i}^{*} a_i})$. By Proposition \ref{prop:4.2}, $\sigma_i$ corresponds to the grafting data \[\{(\overline{\gamma_1},\tau_1), \cdots ,(\overline{\gamma}_{i-1},\tau_{i-1}),(\overline{\gamma}_i,\tau_i),(\delta_i,a_i)\}\] where $\overline{\gamma}_i$ and $\delta_i$ are illustrated as in Figure \ref{fig:23}.
             
  \begin{figure}[H]
   \centering
   \begin{overpic}[width=10cm,clip]{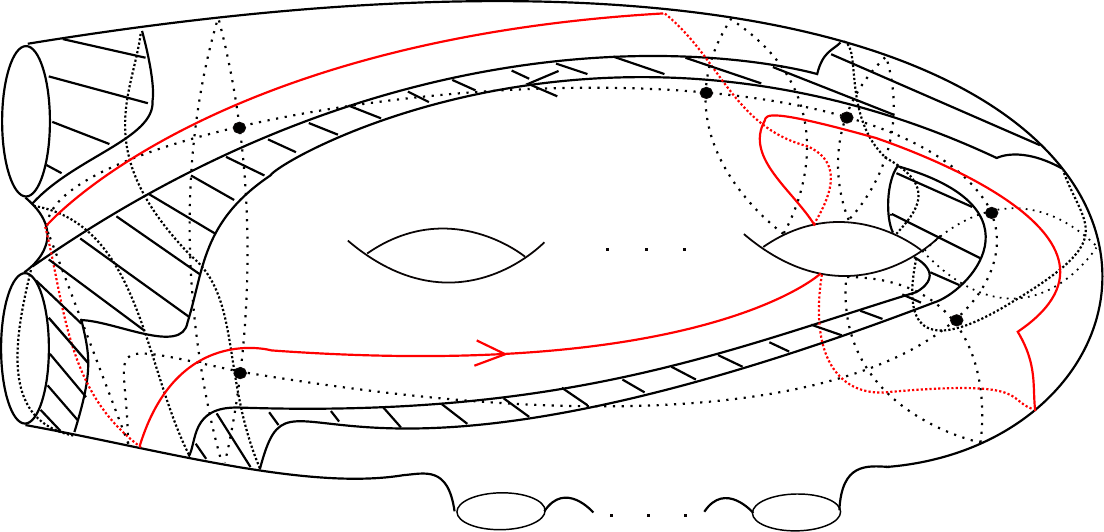}
   \put (140,100) {$\overline{\gamma}_i$}
   \put (130,50) {\color{red} $\delta_i$}
   \put (-15,105) {$\overline{\gamma}_{i_1}$}
   \put (-15,45) {$\overline{\gamma}_{i_2}$}
   \end{overpic}
   \caption{The simple closed curves $\overline{\gamma}_i$ and $\delta_i$ on $\Sigma_{i-1}$. The shaded region is the pair-of-pants with boundary components $\overline{\gamma}_{i_1}$, $\overline{\gamma}_{i_2}$ and $\overline{\gamma}_i$.}
   \label{fig:23}
  \end{figure}
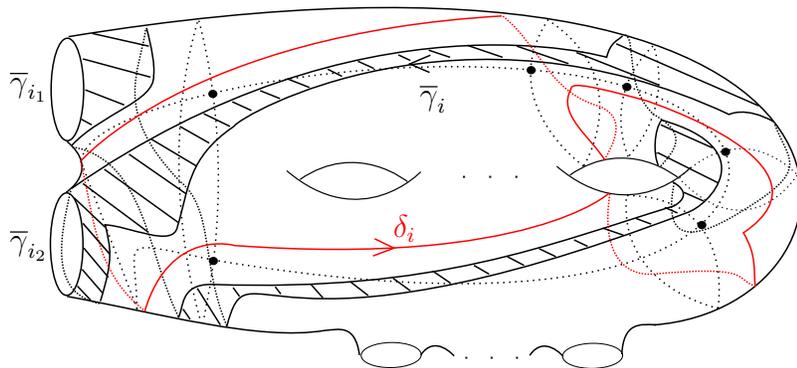
  
  When we cut $\Sigma_{i-1}$ along $\overline{\gamma}_i$, it decomposes into a pair-of-pants and a subsurface $\Sigma_i$. $\delta_i$ is the non-separate simple closed curve on $\Sigma_i$.
 
           \item[(b-2)] Otherwise, i.e., no such pair-of-pants exists: $\Sigma_{i-1}$ is a compact subsurface of $\Sigma$ with genus at least 2. We place $\delta_{i-1}$ in the position on $\Sigma_{i-1}$ illustrated as in Figure \ref{fig:24}.
             
  \begin{figure}[h]
   \centering
   \begin{overpic}[width=10cm,clip]{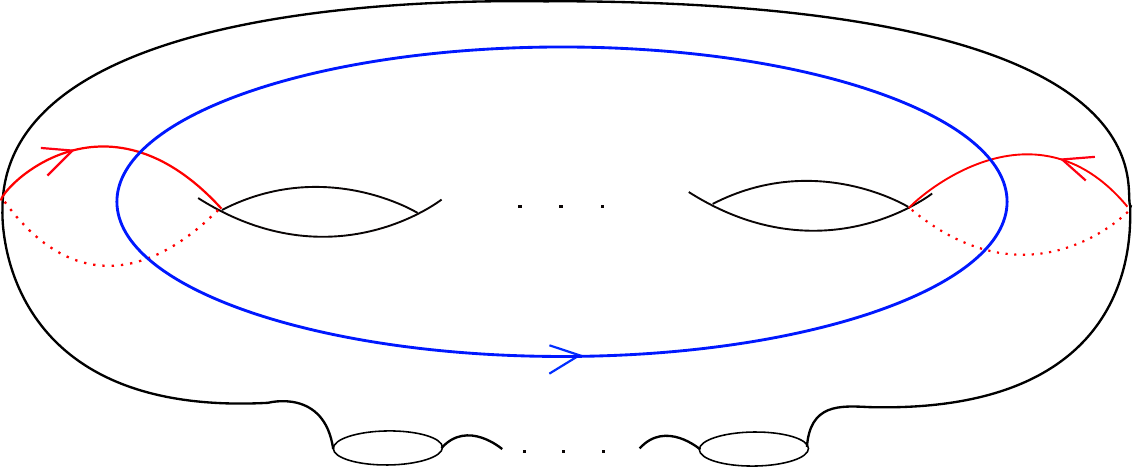}
   \put (250,85) {\color{red} $\delta_{i-1}$}
   \put (20,85) {\color{red} $\delta_{i-1}'$}
   \put (140,35) {\color{blue} $\beta_{i-1}$}
   \end{overpic}
   \caption{The simple closed curves $\delta_{i-1}$,$\delta_{i-1}'$ and $\beta_{i-1}$ on $\Sigma_{i-1}$}
   \label{fig:24}
  \end{figure}            
             
  Let $\delta_{i-1}'$ be a non-separating simple closed curve on $\Sigma_{i-1}$ which does not intersect $\delta_{i-1}$. Let $\beta_{i-1}$ be the non-separating simple closed curves intersecting $\delta_{i-1}$ and $\delta_{i-1}'$ exactly once respectively illustrated in Figure \ref{fig:24}. We obtain a projective structure $\sigma_{i-1}'$ corresponding to grafting data \[\{(\overline{\gamma}_1,\tau_1), \cdots ,(\overline{\gamma}_{i-1},\tau_{i-1}), (\delta_{i-1},a_i a_{i-1}),(\delta_{i-1}',\tau_i)\}\] by grafting along $\delta_{i-1}$ and $\delta_{i-1}'$. Then, we may obtain another projective structure $\sigma_i = \mathbf{Gr}(\sigma_{i-1}',((\beta_{i-1})_R,T^{\rho(\beta_{i-1})}_{a_{i-1}}))$. By Proposition \ref{prop:4.2}, $\sigma_i$ corresponds to the grafting data \[\{(\overline{\gamma}_1,\tau_1), \cdots ,(\overline{\gamma}_{i-1},\tau_{i-1}),(\overline{\gamma}_{i},\tau_i),(\delta_i,a_i)\}\] where $\overline{\gamma}_i$ and $\delta_i$ are illustrated as in Figure \ref{fig:25}.
  
  \begin{figure}[h]
   \centering
   \begin{overpic}[width=10cm,clip]{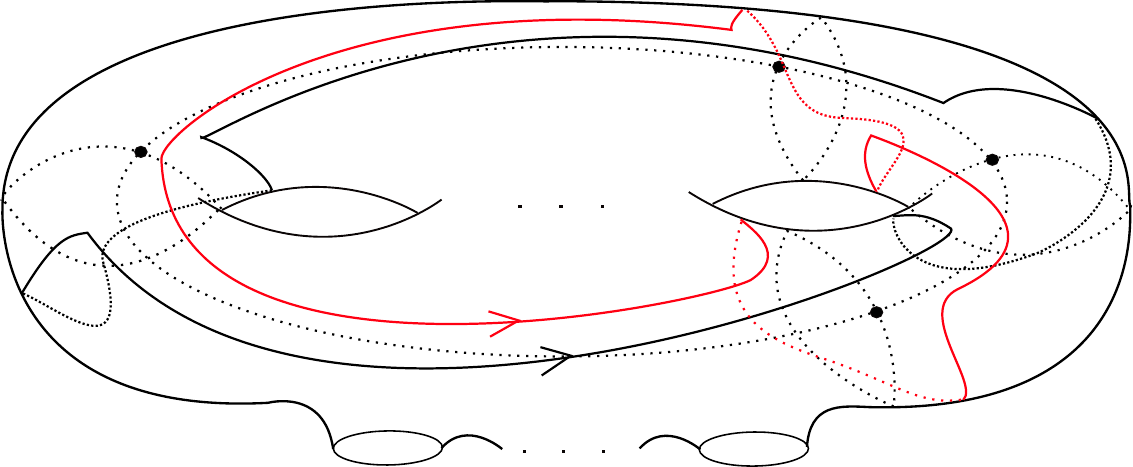}
   \put (160,20) {$\overline{\gamma}_i$}
   \put (110,40) {\color{red} $\delta_i$}
   \end{overpic}
   \caption{The simple closed curves $\overline{\gamma}_i$ and $\delta_i$ on $\Sigma_{i-1}$}
   \label{fig:25}
  \end{figure}      
  
  Cutting the surface $\Sigma_{i-1}$ along $\overline{\gamma_i}$ produces a subsurface $\Sigma_i$. $\delta_i$ is the non-separate simple cosed curve on $\Sigma_i$.
        \end{enumerate}
    \end{enumerate}
  \item[(ii)] Finally, we construct simple closed curves $\overline{\gamma}_{k+1}, \cdots ,\overline{\gamma}_{3g-3}$: $\Sigma_k$ is a compact subsurface of $\Sigma$ with genus 1. $\Sigma_k$ has $3g-3-k$ boundary components. We place $\delta_k$ in the position on $\Sigma_k$ illustrated as in Figure \ref{fig:26}.
  
  \begin{figure}[h]
   \centering
   \begin{overpic}[width=10cm,clip]{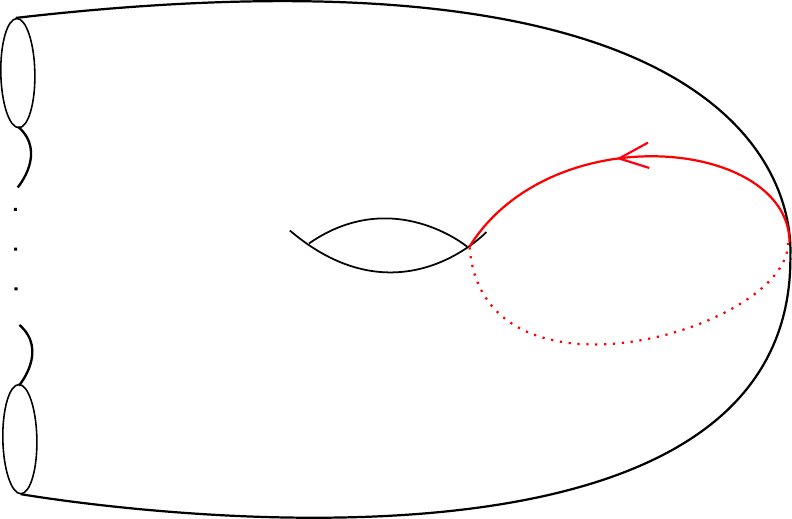}
   \put (200,135) {\color{red} $\delta_k$}
   \end{overpic}
   \caption{The simple closed curves $\delta_k$ on $\Sigma_k$}
   \label{fig:26}
  \end{figure}      
  
  Let $\delta_{k+1}, \cdots \delta_{3g-3}$ and $\beta$ be the simple closed curves on $\Sigma_k$ such that $\delta_{3g-3}$ istopic to $\delta_k$, and $\beta$ intersects each simple close curve $\delta_i$ ($i=k, \cdots,3g-3$) at one point. In Figure \ref{fig:27}, we illustrated the case $\Sigma_k$ has three boundary components.
  
  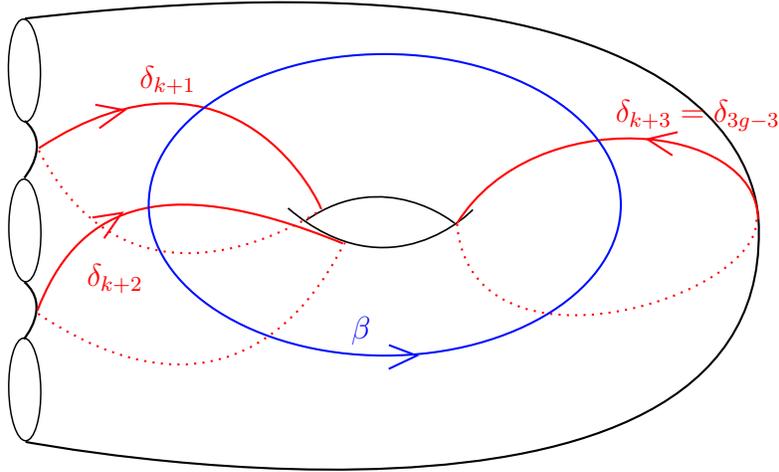
\begin{figure}[h]
   \centering
   \begin{overpic}[width=10cm,clip]{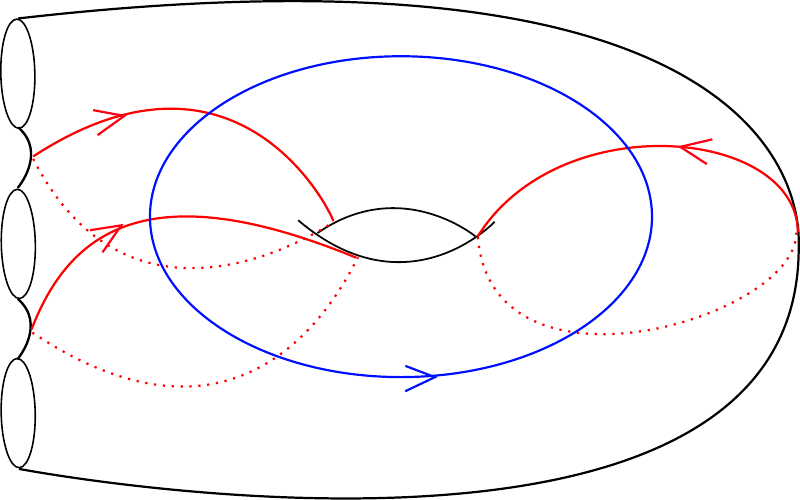}
   \put (230,132) {\color{red} $\delta_{k+3}=\delta_{3g-3}$}
   \put (50,145) {\color{red} $\delta_{k+1}$}
   \put (30,70) {\color{red} $\delta_{k+2}$}
   \put (130,50) {\color{blue} $\beta$}
   \end{overpic}
   \caption{The simple closed curves $\delta_{k+1}$, $\delta_{k+2}$, $\delta_{k+3}$ and $\beta$ on $\Sigma_k$, which has three boundary components}
   \label{fig:27}
  \end{figure}     
  
  We obtain a projective structure $\sigma_k'$ corresponding to grafting data \[\{(\overline{\gamma}_1,\tau_1), \cdots ,(\overline{\gamma}_{k},\tau_k), (\delta_{k+1},\tau_{k+1}), \cdots ,(\delta_{3g-4},\tau_{3g-4}),(\delta_{3g-3},\tau_{3g-3}a_k)\}\] by grafting along $\delta_{k+1}, \cdots , \delta_{3g-3}$. 
Then we may another projective structure $\sigma_{k+1} = \mathbf{Gr}(\sigma_k',(\beta_R,T^{\rho(\beta)}_{a_k}))$. By Proposition \ref{prop:4.2}, the projective structure $\sigma_{k+1}$ corresponds to the grafting data \[\{(\overline{\gamma}_1,\tau_1), \cdots ,(\overline{\gamma}_{k},\tau_k), (\overline{\gamma}_{k+1},\tau_{k+1}), \cdots ,(\overline{\gamma}_{3g-3},\tau_{3g-3})\}.\] $\overline{\gamma}_{k+1}, \cdots ,\overline{\gamma}_{3g-3}$ are simple closed curves on $\Sigma_k$ such that cutting $\Sigma_k$ along $\overline{\gamma}_{k+1}, \cdots ,\overline{\gamma}_{3g-3}$ decomposes into $3g-3-k$ pair-of-pants. Figure \ref{fig:28} illustrates the case $\Sigma_k$ has three boundary components.

 \begin{figure}[h]
   \centering
   \begin{overpic}[width=10cm,clip]{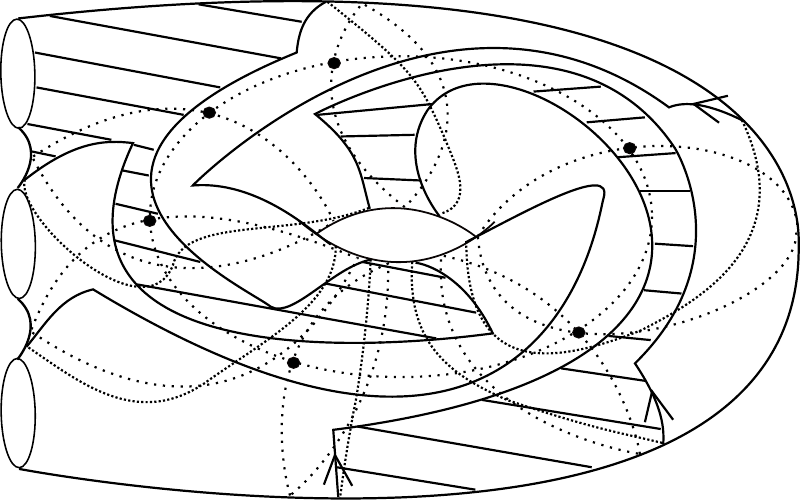}
   \put (240,10) {$\overline{\gamma}_{k+1}$}
   \put (265,140) {$\overline{\gamma}_{k+2}$}
   \put (90,15) {$\overline{\gamma}_{k+3}$}
   \end{overpic}
   \caption{The simple closed curves $\overline{\gamma}_{k+1}$, $\overline{\gamma}_{k+2}$, $\overline{\gamma}_{k+3}=\overline{\gamma}_{3g-3}$ on $\Sigma_k$, which has three boundary components. The shaded region is the pair-of-pants which has boundaries $\overline{\gamma}_{k+3}$ and $\overline{\gamma}_{k+1}$.}
   \label{fig:28}
  \end{figure}
  
 \end{enumerate}
 
 We have constructed a projective structure corresponding to grafting data $(\{\overline{\gamma}_i\},\{\tau_i\})$ such that there is an orientation-preserving diffeomorphism $\phi$ which sends each $\overline{\gamma}_i$ to $\gamma_i$ or $\gamma^{-1}_i$.

 Finally, we count the number of multi-graftings carried out passing from $\mathbf{Gr}_{\eta_1}$ to $\mathbf{Gr}_{\eta_2}$. At most $6$ multi-graftings are required to pass from $\mathbf{Gr}_{\eta_1}$ to $\mathbf{Gr}_{\eta_1'}$. In the latter step, $2(k+1)$ multi-graftings are performed. Therefore, we can pass from $\mathbf{Gr}_{\eta_1}$ to $\mathbf{Gr}_{\eta_2}$ via at most $2(k+1)+6 \le 6g$ multi-graftings.
\end{proof}

\section{The oriented graph $MG(\rho)$}

 Recall that $MG(\rho)$ is the oriented graph such that the vertices are the real proejctive structures with Hitchin holonomy $\rho$, and two vertices $\sigma_1$, $\sigma_2$ are joined by an oriented edge from $\sigma_1$ to $\sigma_2$ if one can obtain $\sigma_2$ form $\sigma_1$ by a multi-grafting.

 We define a (partial) order $\le$ on the set of completely even words as follow.
Given two completely even words $w$ and $w'$, $w \le w'$ if there are a decomposition of $w$ into some positive words $w = w_1 w_2 \cdots w_n$ and completely even words $\tau_1, \cdots, \tau_{n-1}$ such that $w' = w_1 \tau_1 \cdots w_{n-1} \tau_{n-1} w_n$. This order induces a preorder (also denoted by $\le$) on grafting data $\mathscr{GD}$: for $\eta, \eta' \in \mathscr{GD}$, $\eta \le \eta'$ if there is a map $\pi \colon WT(\eta) \to WT(\eta')$ such that $w \le \pi(w)$ for any $w \in WT(\eta)$.

 Recall that we denote by $T^{A}_{w}$ the special $\mathbb{RP}^2$-torus with a hyperbolic holonomy $A$ corresponding to a completely even word $w$. If two completely even words $w$ and $w'$ satisfying $w \le w'$, $T^{A}_{w'}$ is obtained from $T^{A}_{w}$ by grafting along principal simple closed geodesics. This implies that, given two grafting data $\eta_1$ and $\eta_2$ satisfying $\eta_1 \le \eta_2$, we may construct grafting data $\eta_1'$ from $\eta_1$ by grafting such that $WT(\eta_1') = WT(\eta_2)$. As an immediate consequence of the main theorem, we obtain the following generalization.

 \begin{corollary} \label{cor:6.1}
   Let $\eta_1$ and $\eta_2$ be grafting data satisfying $\eta_1 \le \eta_2$. Then we may obtain $\mathbf{Gr}_{\eta_2}$ from $\mathbf{Gr}_{\eta_1}$ by a composition of multi-graftings.
 \end{corollary}

 The preorder $\le$ induces an equivalence relation $\sim$ on $\mathscr{GD}$, defined by $\eta \sim \eta'$ if and only if $\eta \le \eta'$ and $\eta' \le \eta$. The set of equivalence classes $\mathscr{GD} / {\sim}$ then inherits an order (also denoted by $\le$), where $[\eta] \le [\eta']$ if $\eta \le \eta'$.

 Corollary \ref{cor:6.1} allows us to construct a model of $MG(\rho)$ associated with weight types, namely the Hasse diagram associated with the ordered set ($\mathscr{GD}/{\sim}$, $\le$).

 \begin{figure}[h]
   \centering
   \begin{overpic}[width=10cm,clip]{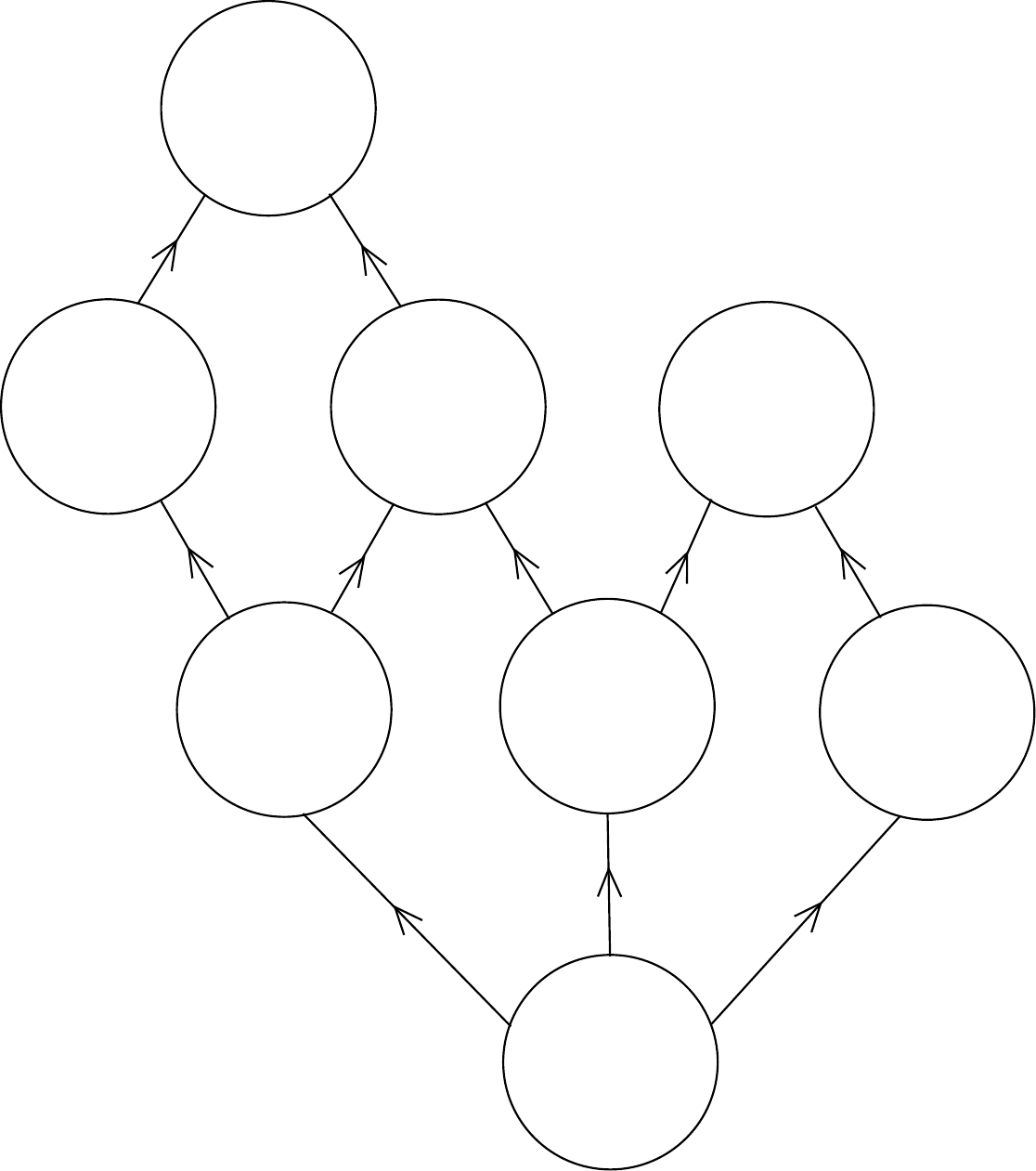}
   \put (165,25) {$1$}

   \put (155,125) {$xyxy$}

   \put (243,123) {$yxyx$}

   \put (71,135) {$xx$}
   \put (71,115) {$yy$}

   \put (199,217) {$xyxy$}
   \put (199,197) {$yxyx$}

   \put (115,225) {$xx$}
   \put (115,210) {$yy$}
   \put (110,195) {$xyxy$}

   \put (17,217) {$xyyx$}
   \put (17,197) {$yxxy$}

   \put (62,305) {$xyyx$}
   \put (62,290) {$yxxy$}
   \put (62,275) {$xyxy$}
   
   \end{overpic}
   \caption{A weight-type model of $MG(\rho)$. Each circle represents the equivalence class of grafting data with the displayed weight type.}
   \label{fig:29}
  \end{figure}
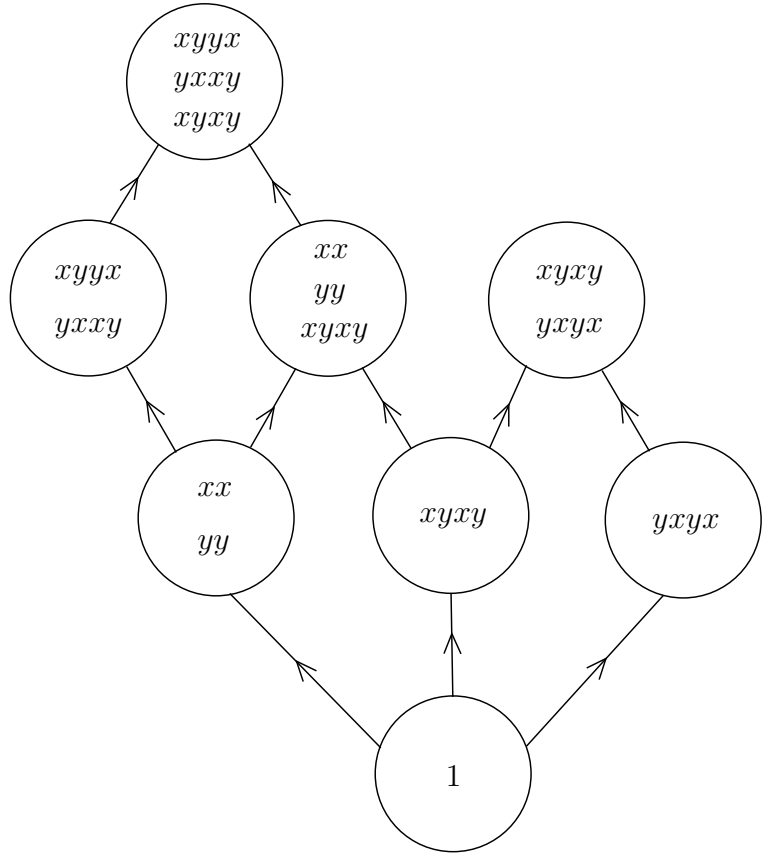


\begin{thebibliography}{Gol1}

  \bibitem[CDF]{CDF:gra}
  G. Calsamiglia-Mendlewicz, B. Deroin and S. Francaviglia, The oriented graph of multi-graftings in the Fuchsian case, Publ. Mat. {\bf 58} (2014), no.~1, 31--46; MR3161507
  
  \bibitem[Cho1]{Choi:admI}
  S. Choi, Convex decompositions of real projective surfaces. I. $\pi$-annuli and convexity, J. Differential Geom. {\bf 40} (1994), no.~1, 165--208; MR1285533
  
  \bibitem[Cho2]{Choi:admII}
  S. Choi, Convex decompositions of real projective surfaces. II. Admissible decompositions, J. Differential Geom. {\bf 40} (1994), no.~2, 239--283; MR1293655

  \bibitem[CG]{CG:hit}
  S. Choi and W.~M. Goldman, Convex real projective structures on closed surfaces are closed, Proc. Amer. Math. Soc. {\bf 118} (1993), no.~2, 657--661; MR1145415

  \bibitem[Gol1]{Goldman:fuc}
  W.~M. Goldman, Projective structures with Fuchsian holonomy, J. Differential Geom. {\bf 25} (1987), no.~3, 297--326; MR0882826

  \bibitem[Gol2]{Goldman:con}
  W.~M. Goldman, Convex real projective structures on compact surfaces, J. Differential Geom. {\bf 31} (1990), no.~3, 791--845; MR1053346

  \bibitem[Gol3]{Goldman:geo}
  W.~M. Goldman, {\it Geometric structures on manifolds}, Graduate Studies in Mathematics, 227, Amer. Math. Soc., Providence, RI, [2022] \copyright 2022; MR4500072

  \bibitem[Ito]{Ito:exo}
  K. Ito, Exotic projective structures and quasi-Fuchsian space. II, Duke Math. J. {\bf 140} (2007), no.~1, 85--109; MR2355068

  \bibitem[Luo]{Luo:multi}
  F. Luo, Some applications of a multiplicative structure on simple loops in surfaces, in {\it Knots, braids, and mapping class groups---papers dedicated to Joan S. Birman (New York, 1998)}, 123--129, AMS/IP Stud. Adv. Math., 24, Amer. Math. Soc., Providence, RI, ; MR1873113

  \bibitem[ST]{ST:exa}
  D.~P. Sullivan and W.~P. Thurston, Manifolds with canonical coordinate charts: some examples, Enseign. Math. (2) {\bf 29} (1983), no.~1-2, 15--25; MR0702731

\end{thebibliography}
\end{document}